\documentclass{amsart}
\usepackage[fontsize=12pt]{scrextend}
\usepackage{m9macros}
\usepackage{mymacros}
\usepackage{tikz}
\usepackage{pgfplots}
\usepackage{tkz-euclide}
\usepackage{tikz-3dplot}
\usetikzlibrary{intersections}
\usetikzlibrary{pgfplots.fillbetween}
\usepackage{graphicx}
\usepackage{float}
\usepackage{enumerate}
\usepackage[english]{babel}
\usepackage{longtable}
\usepackage{tabularx}
\newcolumntype{C}[1]{>{\centering\arraybackslash}p{#1}}
\newcolumntype{L}[1]{>{\raggedright\arraybackslash}p{#1}}
\graphicspath{ {./images/} }
\usepackage[ruled,vlined]{algorithm2e}
\usepackage{cleveref}
\usepackage[utf8]{inputenc}
\usepackage{csquotes}
\usepackage{fancyhdr}
\usepackage{amsthm}
\usepackage{mathtools}
\renewcommand{\S}{\mathds{S}}

\graphicspath{ {figures/} }
\title[Complete system for the in-, circumradius and diameter in $3$-space]{A complete system of inequalities for the diameter, in-, and circumradius in the $3$-dimensional Euclidean space}
\author[R. Brandenberg]{René Brandenberg}
\author[B. González]{Bernardo González Merino}
\author[M. Runge]{Mia Runge}
 \thanks{
This research has been funded by the PID2022-136320NB-I00 project / AEI/10.13039/501100011033/ FEDER, UE}

\begin{document}

\newtheorem{theorem}{Theorem}[section]
\newtheorem{lemma}[theorem]{Lemma}
\newtheorem{proposition}[theorem]{Proposition}
\newtheorem{corollary}[theorem]{Corollary}
\newtheorem{example}[theorem]{Example}
\newtheorem{problem}[theorem]{Problem}
\newtheorem{question}[theorem]{Question}
\newtheorem{conjecture}[theorem]{Conjecture}
\theoremstyle{definition}
\newtheorem{definition}[theorem]{Definition}
\newtheorem{remark}[theorem]{Remark}
\newtheorem*{definition*}{Definition}
\newtheorem*{proposition*}{Proposition}


\begin{abstract}
    We present a complete system of inequalities for the inradius, circumradius, and diameter in the $3$-dimensional Euclidean space. To do so, we prove quasiconcavity of the inradius evaluated over $n$-simplices with a common facet independently of the norm/gauge under consideration.
\end{abstract}
\maketitle

\section{Introduction}
Let $\B_2^n$ denote the $n$-dimensional \cemph{red}{(Euclidean) unit ball} and $\CK^n$ be the set of \cemph{red}{convex bodies} (i.e.~non-empty, convex, and compact sets) in $\R^n$. For any $K\in\CK^n$, let $R(K)$ be the \cemph{red}{circumradius} of $K$, (i.e.~the smallest $\rho\geq 0$ such that a translation of 
a ball of radius $\rho$ covers $K$) and $r(K)$ be the \cemph{red}{inradius} of $K$ (i.e. the largest $\rho\geq 0$ such that a translation of a ball of radius $\rho$ is contained in $K$).  Finally, let $D(K)$ be the \cemph{red}{diameter} of $K$ (i.e.~the maximal length of a segment within $K$). 

The aim of this paper is to describe the range of values that the inradius, circumradius and diameter of $K$ in the $3$-dimensional Euclidean space may achieve. To do so, we compute a \cemph{red}{complete system of inequalities} for those functionals, i.e. a list of inequalities such that if and only if a given 3-tuple of parameters $(r,D,R)$ fulfills all those inequalities, there exists a convex body whose inradius, diameter, and circumradius coincide with those parameters. 

\begin{theorem}
    \label{thm:fulldiagram}
   Let $K\in\CK^3$. Then, 
   \begin{equation}\label{eq:4eineqsKnown}
          2R(K) \geq D(K),\quad \sqrt{3}D(K)\geq\sqrt{8}R(K),\quad D(K)\geq r(K)+R(K),\quad r(K) \geq 0,
   \end{equation}
   and whenever $D(K) \le \sqrt{3}R(K)$ holds true then
   \begin{equation}\label{eq:NewInequality}
          r(K) \geq \frac{D(K)^2\sqrt{3R(K)^2-D(K)^2}}{4R(K)\sqrt{3R(K)^2-D(K)^2}+\sqrt{3}(4R(K)^2-D(K)^2)} \ .
   \end{equation}
    Moreover, \eqref{eq:4eineqsKnown} and \eqref{eq:NewInequality} state a complete system of inequalities for 
    the inradius, diameter, and circumradius in Euclidean 3-space.
\end{theorem}

The new contribution within this theorem 
is the last inequality. We prove the validity of \eqref{eq:NewInequality} and also that it is sharp if and only if $K$ is a 3-simplex 
having at least five diametrical edges (which includes the equilateral 3-simplex) or $K$ is an equilateral triangle. 
A lower bound for the inradius has been derived by Dekster forty years ago \cite[Lem.~1.3 (5)]{Dekster1985}, including the open task of finding a best possible inequality \cite[Rem.~(1)]{Dekster1985}.\footnote{\eqref{eq:NewInequality} has been explicitly determined (and claimed valid and optimal by experimental observations) within the Master thesis \enquote{Complete Systems of Inequalities Describing the Feasible Configurations of Triples of Geometrical Functionals of M.~Horsch, at Technical University of Munich (2019).}} 

This is the very first time that such a complete system of inequalities for a triple of functionals has been derived for the whole family $\CK^3$ of $3$-dimensional convex bodies. So far, only in \cite{HCPSS}, restricting to centrally symmetric convex bodies, such a system could be derived. 

In order to visualize Theorem \ref{thm:fulldiagram}, we consider the mapping
\begin{equation}
\label{eq:bsdiagramdef}
    f:\CK^3\rightarrow[0,1]^2,\quad \quad f(K):=\left(\frac{r(K)}{R(K)},\frac{D(K)}{2R(K)}\right).
\end{equation}

Realize that a complete description of the so-called $(r,D,R)$-Blaschke-Santal\'o diagram $f(\CK^3)$ is equivalent to providing a complete system of inequalities for the inradius, circumradius and diameter. See \Cref{fig:Euclideandiagram} for a sketch of the $(r,D,R)$-diagram \eqref{eq:bsdiagramdef}.

Historically, it was Blaschke who, in 1916, first proposed the question of what values the volume, surface area, and integral mean curvature of three-dimensional convex bodies can have \cite{blaschke}. Later on Santal\'o \cite{santalo} studied complete systems of inequalities for planar sets for triples of geometric functionals (
including the planar analog of \Cref{thm:fulldiagram}). 
Several other authors continued Santaló's work \cite{cifre03,delyon2021,cifre00Dwr,cifre02}
and even different functionals \cite{ftouhi2022complete, ftouhi} or four functionals at the same time \cite{3dimBSdiagram, ting2005extremal} 
have been considered. 

In order to show the fifth inequality in Theorem \ref{thm:fulldiagram}, we prove a quasiconcavity property for the inradius with respect to simplices sharing a common facet. We do so, not restricting to the Euclidean case. For $K,C\in \CK^n$, let $r(K,C)$ denote the \cemph{red}{inradius of $K$ with respect to $C$}, i.e.~the largest $\rho\geq 0$ such that a translation of $K$ contains $\rho C$.  
The set of \cemph{red}{extreme points} of $K$, \ie~those points $p$ in $K$ that are not contained in the convex hull of $K \setminus \set{p}$, 
is denoted by $\ext(K)$. 
\begin{theorem}
\label{thm:minquasiconc}
Let $C\in\CK^n$ be full-dimensional and $p^2,\ldots,p^{n+1}\in\R^n$ 
affinely independent, and $P$ be a convex set contained in the open half-space bounded by the affine hull of $p^2,\ldots,p^{n+1}$. Define $S_p$, $p \in P$,
to be the simplex with vertices $\set{p,p^2,\ldots,p^{n+1}}$. Then there exists 
$p^* \in \ext(P)$ such that $S_{p^*}$ has minimal inradius $r(S_{p^*},C)$ over all simplices $S_{p}, p \in P$.
\end{theorem}
Let us observe that \Cref{thm:minquasiconc} has been proven for the $2$-dimensional Euclidean case \cite{santalo}, but never before in its full generality. This may open the door for obtaining new inequalities 
in higher dimensions, for different gauges, and other combinations of functionals.


The paper is organized as follows. In \Cref{sec:technical}, we collect the definitions and technical results needed throughout the paper. In \Cref{sec:quasiconc}, we prove \Cref{thm:minquasiconc} and in \Cref{sec:euclidean}, we show the new inequality \eqref{eq:NewInequality} in order to prove \Cref{thm:fulldiagram}.
  
  \begin{figure}
    \centering
    \begin{tikzpicture}[scale=4]

     \draw[->] (-0.1,0) -- (1.1,0) node[right] {$\frac{r}{R}$};
     \draw[->] (0,-0.1) -- (0,1.2) node[above] {$\frac{D}{2R}$}
     ;

     \draw[ domain=0:1, smooth, variable=\x, color=black] plot ({\x}, {1})
     ;
     \draw[ domain=0.633:1, smooth, variable=\x] plot ({\x}, {0.5*\x+0.5})
     ;
     \draw[name path=F4, domain=0.333:0.633, smooth, variable=\x] plot ({\x},{0.8165} ) ;
     \draw[ domain={2/sqrt(6)}:{sqrt(3)/2}, smooth, variable=\y] plot ({(4*((\y)^2)*sqrt(3-4*((\y)^2)))/(4*sqrt(3-4*((\y)^2))+sqrt(3)*(4-4*((\y)^2)))},{\y} ); 
     \draw[ domain=0:0.333, dotted, variable=\x] plot ({\x}, {0.8165})
     ;
     \draw[domain=0:0.05, variable=\x] plot ({\x},{sqrt(3)/2});
       \draw[ domain=0:1, dotted, variable=\y] plot ({1}, {\y})
     ;
      \node[label=below:$\phantom{\sqrt{\frac{8}{3}}}1\phantom{\sqrt{\frac{8}{3}}}$] (one) at (1,0) {};
     \draw[ domain=0:0.8165, dotted, variable=\y] plot ({0.333}, {\y})
     ;
     \draw[ domain=0:0.8165, dotted, variable=\y] plot ({0.633}, {\y})
     ;
      \node[label=below:$\phantom{\sqrt{\frac{8}{3}}}\frac{1}{3}\phantom{\sqrt{\frac{8}{3}}}$] (max) at (0.333,0) {};
      \node[label=above left:$1$] (delta) at (0,0.95) {};
       \node[label=below:$\sqrt{\frac{8}{3}}-1$] (delta2) at (0.633,0) {};
       
      \node[label=left:$\frac{\sqrt{3}}{2}$] (rho) at (-0.1,0.9) {};
       \node[label=left:$\frac{2}{\sqrt{6}}$] (rho2) at (-0.1,0.68) {};
 \tkzDefPoint(-0.03,{sqrt(3)/2+0.01}){A}
 \tkzDefPoint(-0.03,{sqrt(2/3)-0.01}){B}
 \tkzDefPoint(-0.13,0.9){C}
 \tkzDefPoint(-0.13,0.7){D}
 \tkzDrawSegment[thin](A,C)
 \tkzDrawSegment[thin](B,D)

     \draw[fill=black] (1,1) circle[radius=0.4pt];

     \draw[fill=black] (0,1) circle[radius=0.4pt];
     \draw[fill=black] (0,{sqrt(3)/2}) circle[radius=0.4pt];
     \draw[fill=black] (0.333,0.8165) circle[radius=0.4pt];
 \draw[fill=black] (0.633,0.8165) circle[radius=0.4pt];

 \end{tikzpicture}
    \caption{The new diagram $ f(\CK^3)$.}
    \label{fig:Euclideandiagram}
\end{figure}
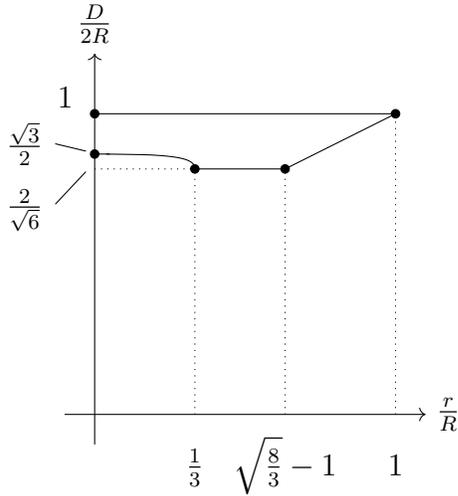

\section{Technical results and definitions}\label{sec:technical}

For any $X\subset \R^n$,  the \cemph{red}{linear, affine,} and \cemph{red}{convex hull} are denoted by $\lin(X)$, $\aff(X)$, and $\conv(X)$, respectively.
The convex hull of two points $x$ and $y$ is called a \cemph{red}{segment} and is abbreviated by $[x,y]$. The open segment is denoted by $(x,y)$. 
The convex hull of $n+1$ affinely independent points is called a \cemph{red}{simplex}. The \cemph{red}{boundary} of $X$ is described by $\bd(X)$ and the \cemph{red}{interior} by $\inte(X)$. Analogously, the \cemph{red}{relative boundary} $\relbd(X)$ and \cemph{red}{relative interior} $\relint(X)$ of $X$ are the boundary and interior of $X$ evaluated  within $\aff(X)$. For any $X,Y\subset\R^n$ and $\rho\in\R$ let $X+Y:=\{x+y: x\in X, y\in Y\}$ be the \cemph{red}{Minkowski sum} of $X$ and $Y$ and $\rho X:=\{\rho x: x\in X\}$ the $\rho$-\cemph{red}{dilatation} of $X$. We abbreviate $\set{x}+Y=:x+Y$ and $(-1)X=:-X$.
The \cemph{red}{support function} of $K \in \CK^n$ is defined as $h_K(\cdot): \R^n \to \R $, $h_K(a):=\max_{x\in K}a^Tx$ and 
the \cemph{red}{polar} as $K^\circ := \{x\in\R^n:x^Ty\leq h_C(y)\text{ for all } y \in K\}$.

The \cemph{red}{circumradius} of $K \in \CK^n$ 
is defined as 
\[
 R(K):=\min\{\rho\geq 0: \exists t \in \R^n \text{ such that } K\subset t+\rho \B^n_2\}
\]
and the \cemph{red}{diameter} as the longest segment in $K$:
\begin{equation*}
    D(K):=\max_{x,y \in K} \norm{x-y}_2. 
\end{equation*}

The \cemph{red}{inradius} of $K \in \CK^n$ with respect to $C\in\CK^n$ is defined as
\[
  r(K,C):=\max\{\rho\geq 0: \exists t \in \R^n \text{ such that } t+ \rho C\subset K\}
\]
and we abbreviate $r(K):=r(K,\B^n_2)$. 
A set $t+r(K,C)C$, $t\in\R^n$, which is contained in $K$ is called an \cemph{red}{inball} of $K$.
One may recognize that $R(K)=R(\ext(K))$ and $D(K)=D(\ext(K))$ (c.f.~\cite{bonnesenfenchel}), which in particular means that the diameter of a polytope, and more specifically a simplex, is attained between two of its vertices.

 One of the first inequalities found, relating these functionals, is Jung's inequality, bounding the diameter from below by the circumradius \cite{Jung1901}. For $K\in\CK^n$ one has 
    \begin{equation}
    \label{eq:jungeucl}
        R(K)\sqrt{\frac{2(n+1)}{n}}\leq D(K). 
    \end{equation}
We use the following notation for \cemph{red}{hyperplanes}:  for $a\in\R^n\setminus\set{0}$ and $\beta\in \R$, we write $H^{ =}_{(a,\beta)}:=\set{x\in\R^n:a^{\top}x = \beta}$ and the according \cemph{red}{half-spaces} are denoted analogously using "$\leq$" and "$\geq$". 


For $K,C\in \CK^n$ we say that $K$ is \cemph{red}{optimally contained} in $C$ if $K\subset C$ and $r(C,K)=1$, which is abbreviated by $K\optc C$.
A proof for the following characterization of optimal containment can be found in \cite{NoDimIndep}.

 \begin{proposition}
    \label{prop:opt}
    Let $K,C\in \CK^n$. Then $K\optc C$ if and only if
    \begin{itemize}
        \item[i)] $K\subset C$ and
        \item[ii)] for some $k\in\{2,\hdots,dim(C)+1\}$, there exist $p^1,\hdots,p^k\in \relbd(K)\cap \relbd(C)$ and half-spaces $H^{\leq}_{(a^{i},(a^i)^Tp^i)}$ supporting $C$ at $p^{i}$ with $a^1,\hdots,a^k\in \ext(C^{\circ})\setminus\{0\}$, affinely independent, such that $0\in \conv(\{a^1,\hdots,a^k\})$.
    \end{itemize}

\end{proposition}
\begin{remark}
    \label{rem:opteuclidean}
    Note that in the Euclidean case $C=\B^n_2$, the boundary points $p^i$ and outer normals $a^i$ in \Cref{prop:opt} ii) coincide. Thus, the condition can be expressed as $0\in \conv(\{p^1,\hdots,p^k\})$. Moreover, in this case, if all the $p^i$ are contained in a half-space with 0 in its boundary, then already the convex hull of the points of $K$ in the 
   bounding hyperplane
    must be optimally contained in $\B^n_2$.
\end{remark}

The following proposition from \cite{3dimBSdiagram} shows that the diagram $f(\CK^n)$ 
is star-shaped with respect to the vertex $f(\B^n_2)=(1,1)$. Thus, it suffices to describe the boundaries of the diagram to show the completeness of such a system of inequalities.  
\begin{proposition}
    \label{prop:starshape}
    Let $K \in  \CK^n$ be such that $K\optc \B^n_2$. Then,
    \begin{equation*}
         f ((1 -\lambda)K +\lambda \B^n_2) = (1-\lambda)f(K)+\lambda f(\B^n_2),
    \end{equation*}
for every $\lambda \in [0,1]$.
\end{proposition}

\section{Quasiconcavity of the inradius over a moving vertex of a simplex}
\label{sec:quasiconc}
To prepare the proof of \Cref{thm:minquasiconc}, we first consider the description of points in two simplices sharing a facet. For $k\in\N$ we use the notation $[k]:=\set{1,2,\ldots,k}$.
\begin{lemma}
    \label{lem:coeffquotient}
    Let 
    $K_0=\conv\left(\set{p^0,p^2,\ldots,p^{n+1}}\right)$ and $K_1=\conv\left(\set{p^1,p^2,\ldots,p^{n+1}}\right)$ be full-dimensional simplices such that $p^0$ and $p^1$ are contained in the same open half-space defined by $\aff\left(\set{p^2,\ldots,p^{n+1}}\right)$. Furthermore, let 
    \begin{align*}
        v&=\lambda_1 p^0 + \sum_{i=2}^{n+1}\lambda_{i}p^i \in K_0 \\
        w&=\mu_1p^1 + \sum_{i=2}^{n+1}\mu_{i}p^i \in K_1
        \end{align*}
        with coefficients $\lambda_i,\mu_i\geq 0$, $i\in[n+1]$,
        \begin{align*}
             \sum_{i=1}^{n+1} \lambda_i =\sum_{i=1}^{n+1} \mu_i = 1,
    \end{align*}
    such that $[v,w]$ is parallel to $\aff\left(\set{p^2,\ldots,p^{n+1}}\right)$. Then, the ratio $\frac{\lambda_1}{\mu_1}$ only depends on the set $\set{p^0,p^1,p^2,\ldots,p^{n+1}}$, but not on the positions of $v,w$. 
\end{lemma}
\begin{proof}
   Since the coefficients $\lambda_i,\mu_i$, $i\in[n+1]$, are invariant under translation and rotation of $K_0$ and $K_1$, we may assume that $p^i_1=0$ for $i=2,\ldots,n+1$. This implies $v_1=w_1$ and $\lambda_1=\frac{v_1}{p^0_1}$ and $\mu_1=\frac{w_1}{p^1_1}$. It follows $\frac{\lambda_1}{\mu_1}=\frac{p^1_1}{p^0_1}$, which is independent of the positions of $v$ and $w$. 
\end{proof}

\begin{lemma}
\label{lem:quasiconc}
Let $C\in\CK^n$ be full-dimensional and $p^0,p^1,p^2,\ldots,p^{n+1}\in\R^n$ such that the points $p^2,\ldots,p^{n+1}$ are affinely independent and $p^0$ and $p^1$ lie in the same open half-space bounded by $\aff\left(\set{p^2,\ldots,p^{n+1}}\right)$.  Define 
$K_\alpha:=\conv\left(\set{(1-\alpha)p^0 +\alpha p^1,p^2,\ldots,p^{n+1}}\right)$ for $\alpha\in[0,1]$ and 
assume $r(K_0,C)=r(K_1,C)$. Then, 
  \begin{equation*}
        r(K_{\alpha},C)\geq r(K_1,C),\quad \alpha\in[0,1]. 
    \end{equation*}
\end{lemma}
\begin{proof}
    Let $r:= r(K_0,C)=r(K_1,C)$. Then for every $v\in C$ there exist coefficients $\lambda_{i,v}$, $\mu_{i,v}$, $i\in[n+1]$, and translations $c,d\in\R^n$ fulfilling 
     \begin{align}
          \lambda_{1,v}p^0 +\sum_{i=2}^{n+1} \lambda_{i,v} p^i &= rv +c  \label{eq:c1}\\
    \sum_{i=1}^{n+1} \lambda_{i,v}&=1 \label{eq:c2}\\
    \lambda_{i,v}&\geq 0 \label{eq:c3}
     \end{align}
     and 
     \begin{align}
          \sum_{i=1}^{n+1} \mu_{i,v} p^i &= rv +d \label{eq:d1}\\
    \sum_{i=1}^{n+1} \mu_{i,v}&=1 \label{eq:d2}\\
    \mu_{i,v}&\geq 0. \label{eq:d3}
     \end{align}
     Now we show that for every $\alpha\in[0,1]$, there exists a translation $e\in[c,d]$, such that $e+rC\subset K_{\alpha}$, by finding a feasible solution of 
\begin{align}
          \epsilon_{1,v}((1-\alpha)p^0+\alpha p^1) + \sum_{i=2}^{n+1} \epsilon_{i,v} p^i &= rv +e \label{eq:e1}\\
    \sum_{i=1}^{n+1} \epsilon_{i,v}&=1 \label{eq:e2}\\
    \epsilon_{i,v}&\geq 0, \quad i\in[n+1] \label{eq:e3}
\end{align}
for any $v\in C$. 
Since $K_0, K_1$ are full-dimensional simplices and $c+rC\optc K_0$ as well as $d+rC\optc K_1$, we know by \Cref{prop:opt} that both, $c+rC$ and $d+rC$, touch the common facet $\conv\left(\set{p^2,\ldots,p^{n+1}}\right)$ of $K_0$ and $K_1$. Thus, it follows that $[c,d]$ and therefore $[rv+c,rv+d]$ for every $v\in C$ are parallel to this facet. Hence, we can apply \Cref{lem:coeffquotient} and obtain $\frac{\lambda_{1,v}}{\mu_{1,v}}= \frac{\lambda_{1,w}}{\mu_{1,w}}=:\kappa$ for all $v,w\in C$. Moreover, since $\lambda_{1,v},\mu_{1,v}\neq 0$, it follows $\kappa\notin\set{0,\infty}$.
Now, for every $v\in C$ we define 
\begin{align*}
    \beta:&= \frac{1}{1+ \frac{1-\alpha}{\alpha}\kappa}  \in (0,1)\\
    e:&=(1-\beta)c+\beta d\\
    \epsilon_{i,v}:&= (1-\beta)\lambda_{i,v}+\beta \mu_{i,v}, \quad i\in[n+1] .
\end{align*}
Then, \eqref{eq:e2} directly follows from \eqref{eq:c2}, \eqref{eq:d2} and \eqref{eq:e3} directly from \eqref{eq:c3}, \eqref{eq:d3}. Moreover, from multiplying \eqref{eq:c1} by $(1-\beta)$ and \eqref{eq:d1} by $\beta$, one obtains
\begin{align*}
    (1-\beta)\lambda_{1,v}p^0 + \beta \mu_{1,v}p^1 + \sum_{i=2}^{n+1} \epsilon_{i,v} p^i = rv +e
\end{align*}
for every $v\in C$. 
Thus, it remains to show that we got the correct coefficient for $(1-\alpha)p^0+\alpha p^1$ to show \eqref{eq:e1}:
\begin{align*}
    \epsilon_{1,v}\left((1-\alpha)p^0+\alpha p^1\right)&=\left((1-\beta)\lambda_{1,v}+\beta \mu_{1,v}\right)\left((1-\alpha)p^0+\alpha p^1\right)\\
    &=\left(\frac{\frac{\mu_{1,v}}{\lambda_{1,v}}\cdot\frac{1-\alpha}{\alpha}}{1+ \frac{\mu_{1,v}}{\lambda_{1,v}}\cdot\frac{1-\alpha}{\alpha}}\lambda_{1,v}+\frac{1}{1+ \frac{\mu_{1,v}}{\lambda_{1,v}}\cdot\frac{1-\alpha}{\alpha}} \mu_{1,v}\right)\left((1-\alpha)p^0+\alpha p^1\right)\\
    &=\frac{\frac{\mu_{1,v}}{\lambda_{1,v}}\cdot\frac{1-\alpha}{\alpha}}{1+ \frac{\mu_{1,v}}{\lambda_{1,v}}\cdot\frac{1-\alpha}{\alpha}} \left(\lambda_{1,v}(1-\alpha)+\frac{\lambda_{1,v}}{\mu_{1,v}}\cdot\frac{\alpha}{1-\alpha}\cdot\mu_{1,v}(1-\alpha)\right)p^0 \\ 
    \quad &+\frac{1}{1+ \frac{\mu_{1,v}}{\lambda_{1,v}}\cdot\frac{1-\alpha}{\alpha}} \left(\frac{\mu_{1,v}}{\lambda_{1,v}}\cdot\frac{1-\alpha}{\alpha}\cdot\lambda_{1,v}\alpha+\mu_{1,v}\alpha\right)p^1
    \\ 
    &=(1-\beta)\lambda_{1,v}p^0 + \beta \mu_{1,v}p^1
\end{align*}
    This proves \eqref{eq:e1} and therefore that $e+rC\subset K_{\alpha}$, which implies $r(K_{\alpha},C)\geq r$.
\end{proof}

We are not restricted to the case where the inradii of $K_0$ and $K_1$ coincide. \Cref{lem:quasiconcsegment} shows that the smallest inradius is attained at the boundary of the segment $[p^0,p^1]$. 

\begin{lemma}
    \label{lem:quasiconcsegment}
    Let $C\in\CK^n$ be full-dimensional and $p^0,p^1,\ldots,p^{n+1}\in\R^n$ such that $p^2,\ldots,p^{n+1}$ are affinely independent and $p^0$ and $p^1$ lie in the same open half-space bounded by the hyperplane $\aff\left(\set{p^2,\ldots,p^{n+1}}\right)$. 
Again, define $K_\alpha:=\conv\left(\set{(1-\alpha)p^0 +\alpha p^1,p^2,\ldots,p^{n+1}}\right)$ 
for $\alpha\in[0,1]$. Then, 
\begin{equation*}
    r(K_\alpha,C)\geq \min(\set{r(K_0,C),r(K_1,C)}).
\end{equation*}
\end{lemma}

\begin{proof}
 By continuity of the inradius with respect to the Hausdorff norm, the mapping from $\alpha$ to the inradius of $K_\alpha$ is continuous. Assume, there is an $\alpha^*\in[0,1]$ such that $r(K_{\alpha^*},C)<\min(\set{r(K_0,C),r(K_1,C)})$ and without loss of generality that $r(K_0,C)\geq r(K_1,C)$. Then there exists an $\alpha\in[0,\alpha^*]$ such that $r(K_{\alpha},C)=r(K_1,C)$, and we can apply \Cref{lem:quasiconc} to obtain $r(K_{\alpha^{*}},C)\geq r(K_1,C)$, a contradiction. 
\end{proof}
\begin{proof}[Proof of \Cref{thm:minquasiconc}]
Consider any $p\in P$ which minimizes the inradius of these simplices and that has the property that its face\footnote{{See, e.g., \cite{Schneider93} for the theory of faces for general convex bodies.}} $F(p)$ is of minimal dimension over the set of such points. If $p$ is extreme, there is nothing to show. Otherwise, let $p^0$ be any extreme point of $F(p)$ and $p^1 \in \relbd(F(p))$ such that $p \in [p^0,p^1]$. From \Cref{lem:quasiconcsegment} we now obtain that at least one of the two simplices $S_{p^i}, i=0,1$ fulfills $r(S_{p^i},C) = r(S_p,C)$, in contradiction to our assumption that $F(p)$ is a minimal dimensional  with this property. 
\end{proof}
To obtain geometric inequalities with the help of \Cref{thm:minquasiconc}, we will compare simplices that share a facet and are optimally contained in $C$. 

In general, the convex hull $P$ of points on the boundary of $C$, our choices for the changing vertex, may not belong to the boundary of $C$ itself. \Cref{cor:bdarea} below reveals that we can compare the inradii if we choose 
the last vertex on a part of the boundary of $C$ which lies in the projection of this convex set $P$ onto the boundary with a center of projection being a vertex of the shared facet of the simplices (\cf~\Cref{fig:projections}). 
For $\set{q^1,\ldots,q^k}\subset\bd(C)$ and $P:=\set{p^1,\ldots,p^m}\subset\bd(C)$  we say that $q$ belongs to the \cemph{red}{$C,P$-convex hull} of $\set{q^1,\ldots,q^k}$, which we denote by $\conv_{C,P}\left(\set{q^1,\ldots,q^k}\right)$, if 
\begin{equation*}
    q \in \conv \left( \bigcup_{j=1}^{m} \set{ p^j+\sum_{i=1}^{k} \alpha_i(q^i-p^j ):\alpha_i\geq 0, i\in[k], \sum_{i=1}^k \alpha_i\geq 1 }\right) \cap \bd(C). 
\end{equation*}
If $P$ consists of a single point $p$, we abbreviate $\conv_{C,\set{p}} =: \conv_{C,p}$.

\begin{figure}[ht]
    \centering
    \tdplotsetmaincoords{0}{0}
\begin{tikzpicture}[tdplot_main_coords, scale=3]

\def\R{1}
\def\D{1.66}

\coordinate (p3) at (-0.926, 0.3778, 0); 
\coordinate (p4) at (0, -1, 0);
\coordinate (p3m)  at (0.926, 0.3778, 0);
\coordinate(z) at (0,0,0);

\shade[ball color=blue!10, opacity=0.2] (0,0,0) circle (\R);

 \tkzDefPointBy[reflection = over p3--z](p4)
 \tkzGetPoint{p4m}

  \tkzDefMidPoint(p3,p3m)
 \tkzGetPoint{m3}

  \tkzDefMidPoint(p4,p4m)
 \tkzGetPoint{m4}
\path[name path=e1](m3) ellipse[x radius=0.926, y radius=0.15];

\path[name path=e2, rotate around={68:(m4)}] (m4) ellipse[x radius=0.926, y radius=0.15];
 
\path[name intersections={of=e1 and e2, by={i1,i2,i3,i4}}];

 \tkzInterLL(p3,p3m)(p4,p4m)
 \tkzGetPoint{p1s}
\tkzDrawSegment[dotted](p4,p3)
\tkzDrawSegment[dotted](i2,p3)
\tkzDrawSegment[dotted](p4,i2)
\tkzDrawSegment[dotted,dred](p4,i3)
\tkzDrawSegment[dotted,dblue](i3,p3)
\tkzDrawSegment[dotted,dgreen](i2,i3)

\tkzDefPointOnLine[pos=0.45](p4,p3)
 \tkzGetPoint{m5}
\path[name path=e3] (m5) circle[radius=1];
\path[name intersections={of=e3 and e1, by={h1,h2,h3,h4}}];
\path[name path=e3, rotate around={292:(m5)}] (m5) ellipse[x radius=0.7, y radius=0.7];

\path[name intersections={of=e3 and e2, by={k1,k2,k3,k4}}];
\path[name intersections={of=e3 and e1, by={h1,h2,h3,h4}}];
\tkzDrawPoint(h2)
\tkzDrawSegment[dotted,dred](p4,h2)
\tkzDrawSegment[dotted,dblue](h2,p3)
\tkzDrawSegment[dotted,dgreen](i2,h2)



\filldraw (p3) circle (0.5pt) node[above left] {$p^3$};
\filldraw(p4) circle (0.5pt) node[below] {$p^4$};    
\filldraw(i2) circle (0.5pt) node[above right] {$p^2$};



\draw[thick, dblue]    (h2) to[out=-5,in=-175] (i3);
\draw[thick, dgreen]    (h2) to[out=-10,in=-170] (i3);
\draw[thick, dred]   (h2) to[out=10,in=170] (i3);
\filldraw(i3) circle (0.5pt) node[below right] {$q^2$};
\tkzDrawPoint(h2)
\tkzLabelPoint[below left](h2){$q^1$}
\end{tikzpicture}
\caption{The $C,p^i$-convex hulls of $\set{q^1,q^2}$
, $i=2,3,4$, are depicted by the three thick colored lines.}
\label{fig:projections}
\end{figure}
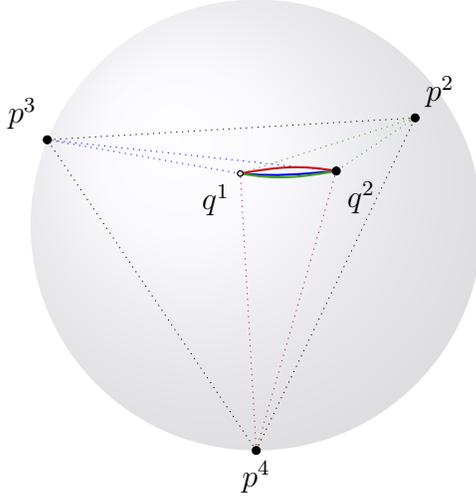
The following lemma and the last part of \Cref{cor:bdarea} play a key role in proving the equality case of \eqref{eq:NewInequality}. 
\begin{lemma}
    \label{lem:inballsimplex}
    Let $C\in\CK^n$ be full-dimensional and smooth, and let $S\in\CK^n$ be a full-dimensional simplex. Then, for any $K\in\CK^n$ with $S\subsetneq K$ we have $r(S,C)<r(K,C)$.
\end{lemma}
\begin{proof}
    Due to \Cref{prop:opt}, the inball of a simplex $S$ must always touch all facets of $S$ and in case the inball is smooth it can only touch the facets in their relative interior.
    However, at least one facet-defining hyperplane of $S$ does not support $K$, which, again by to \Cref{prop:opt}, immediately shows that the inball of $S$ cannot be optimally contained in $K$, 
    implying $r(S,C)<r(K,C)$.
\end{proof}

\begin{corollary}  \label{cor:bdarea}
Let $C\in\CK^n$ be full-dimensional and $\set{p^2,\ldots,p^{n+1}}\subset \bd(C)$ be affinely independent. Furthermore, let $\set{q^1,\ldots,q^k}\subset \bd(C)$ be contained in the same open half-space bounded by the hyperplane $\aff\left(\set{p^2,\ldots,p^{n+1}}\right)$. Then, for $S:=\conv\left(\set{p^1,p^2,\ldots,p^{n+1}}\right)$ with $p^1 \in \conv_{C,\set{p^2,...p^{n+1}}}$ and $S_i:=\conv\left(\set{q^i,p^2,\ldots,p^{n+1}}\right)$
we have 
\begin{equation*}
   r(S,C)\geq \min_{i\in[k]} r(S_i,C).
\end{equation*}
Moreover, if $C$ is smooth and $p^1$ is not contained in $\conv\left(\set{q^1,\ldots,q^k}\right)$, we have $r(S,C)>\min_{i\in[k]} r(S_i,C) $.
\end{corollary}
\begin{proof}
    Essentially, it suffices to show that there exists some 
    $q\in \conv\left(\set{q^1,\ldots,q^k}\right)\cap S$. 
    Given such a $q$, it is obvious
    that $r(S,C)\geq r(\conv\left(\set{q,p^2,\ldots,p^{n+1}}\right),C)$ and 
    $q \in \conv\left(\set{q^1,\ldots,q^k}\right)$ enables us to apply \Cref{thm:minquasiconc} to obtain $r(\conv\left(\set{q,p^2,\ldots,p^{n+1}}\right),C)\geq \min_{i\in[k]} r(S_i,C)$. 
     Note that we have $r(S,C)>r(\conv\left(\set{q,p^2,\ldots,p^{n+1}}\right),C)$  if $C$ is smooth and  $q\in\inte(C)$ by \Cref{lem:inballsimplex}. 

   According to our choice of $p^1$ there exist $\lambda_j, \alpha_{i,j}\geq0$, $j\in\set{2,\ldots,,n+1}$ and $i\in[k]$, 
   with $\sum_{j=2}^{n+1}\lambda_j=1$   and $\sum_{i=1}^{k}\alpha_{i,j}\geq 1$ such that 
   \begin{equation*}
       p^1=\sum_{j=2}^{n+1} \lambda_j \left( p^j + \sum_{i=1}^k \alpha_{i,j}(q^i-p^j) \right) \in \bd(C). 
   \end{equation*}
We define
\begin{align*}
\beta_i &:= \frac{\sum_{j=2}^{n+1}\lambda_j \alpha_{i,j}}{\sum_{m=2}^{n+1}\lambda_m \sum_{l=1}^k \alpha_{l,j} }, \quad i \in [k] \\
\mu_1&:=\frac{1}{\sum_{m=2}^{n+1}\lambda_m \sum_{l=1}^k \alpha_{l,j} },  \\
\mu_j&:=\frac{\lambda_j\left(\sum_{i=1}^{k} \alpha_{i,j}-1\right)}{\sum_{m=2}^{n+1}\lambda_m \sum_{l=1}^k \alpha_{l,j} }, \quad j \in \set{2,\ldots,n+1}. 
\end{align*}
Then, we have $\beta_i, \mu_j\in[0,1]$ for $i\in[k]$ and $j\in\set{2,\ldots,n+1}$ and $\sum_{i=1}^k \beta_i=\sum_{j=1}^{n+1}\mu_j=1$. 

Moreover, 
\begin{align*}
   q:&=\sum_{j=1}^{n+1}\mu_j p^j= \mu_1 \left(\sum_{j=2}^{n+1}\lambda_j\left(p^j+\sum_{i=1}^k \alpha_{i,j}(q^i-p^j) \right)\right) +\sum_{j=2}^{n+1}\mu_j p^j\\ 
    &= \mu_1\left(\sum_{j=2}^{n+1}\lambda_j\sum_{i=1}^k \alpha_{i,j}q^i \right)+ \sum_{j=2}^{n+1} \left(\mu_1\lambda_j\left(1-\sum_{i=1}^{k}\alpha_{i,j}\right)+\mu_j\right)p^j\\
    &=\sum_{i=1}^k \beta_i q^i \in \conv\left(\set{q^1,\ldots,q^k}\right), 
\end{align*}
 concluding the proof.
\end{proof}

\section{Proof of the main result}
\label{sec:euclidean}
To prove Theorem \ref{thm:fulldiagram}, we need to show \eqref{eq:NewInequality}. To do so, we aim to minimize the inradius, given a fixed diameter and circumradius. 
For simplicity, we abbreviate $\B:=\B_2^3$ for the 3-dimensional Euclidean ball and write $\S:=\bd(\B)$ for the corresponding sphere.
\begin{remark}
    \label{rem:sqrt3}
     We know that for 
$D \in [\sqrt{3},2]$, there exist planar convex sets $K\in\CK^3$ with $D(K)=D$, $R(K)=1$ and $r(K)=0$ \cite{santalo} and from \eqref{eq:jungeucl} that $D(K)\geq \sqrt{\frac{8}{3}}$ for all $K\in\CK^3$ with $R(K)=1$.  Furthermore, if $S:=\conv(\set{p^1,p^2,p^3,p^4})\optc \B$ with 
$D(S)<\sqrt{3}$ then, by \Cref{rem:opteuclidean} and \eqref{eq:jungeucl}  (applied for the planar case), $p^1,p^2,p^3,p^4$ are affinely independent and contained in $\S$.
\end{remark}

In the following, we fix a diameter $D\in\left[\sqrt{\tfrac{8}{3}},\sqrt{3}\right)$ and find the smallest inradius a full-dimensional simplex with this diameter and circumradius 1 can have. Later, we will show why it suffices to consider simplices.

Let $S:=\conv(\set{p^1,p^2,p^3,p^4})\optc \B$ with 
$D(S)=D$. Then, $\set{p^1,p^2,p^3,p^4}\subset \S$ by \Cref{rem:opteuclidean}. Since a simplex attains its diameter with one of its edges and $\B$ is invariant under rotations, we may assume

\begin{align}
\label{eq:p3p4def}
   p^3=\begin{pmatrix}
       -\sqrt{D^2- \frac{D^4}{4}} \\ \tfrac{D^2}{2}-1 \\0
    \end{pmatrix} \quad \text{and} \quad
    p^4=\begin{pmatrix}
        0 \\ -1 \\0
    \end{pmatrix}.
\end{align}
We define the small circles 
\begin{equation*}
    \Gamma_i:=\set{x\in \S: \norm{x-p^i}=D}, \quad i\in[4].
\end{equation*}

Note that for $p\in\S$,
\begin{equation}
\label{eq:distancehalf-space}
\begin{split}
      \set{x\in\S: \norm{x-p} =D}=\S \cap H^{ =}_{\left(-p,\frac{D^2}{2}-1\right)}
\end{split}
\end{equation}
and analogously for "$\leq$" and "$\geq$".

\begin{remark}
    \label{rem:p3p4}
    The small circles $\Gamma_3$ and $\Gamma_4$ intersect in two points, namely
\begin{align*}
    p^1_*:=\begin{pmatrix}
        \frac{D^3-2D}{2\sqrt{4-D^2}}\\ \frac{D^2}{2}-1 \\ \sqrt{1- \frac{(D^3-2D)^2}{4(4-D^2)}-\left(\frac{D^2}{2}-1\right)^2}
    \end{pmatrix} \quad \text{and} \quad
     p^2_*:=\begin{pmatrix}
        \frac{D^3-2D}{2\sqrt{4-D^2}}\\ \frac{D^2}{2}-1 \\ -\sqrt{1- \frac{(D^3-2D)^2}{4(4-D^2)}-\left(\frac{D^2}{2}-1\right)^2}
    \end{pmatrix}
\end{align*}
(c.f.~\Cref{fig:eucl-basic}). The points coincide if and only if $D=\sqrt{3}$. 
Moreover, keeping $p^4$ and investigating the path $p^3$ would take under the above definition when reducing the diameter", with the second coordinate of $p^3$ being $y\leq\frac{D^2}{2}-1$, the first coordinate of the intersection points would 
become $\sqrt{\frac{1+y}{1-y}}\left(\frac{D^2}{2}-1\right)$. 
This expression is increasing in $y$,
which means that the two intersection points are maintained when $p^3$ is moved towards $p^4$.
\end{remark}

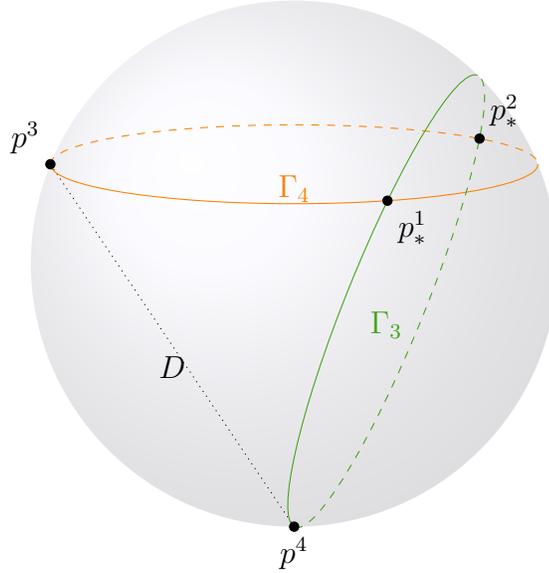
\begin{figure}
    \centering

    \tdplotsetmaincoords{0}{0}
\begin{tikzpicture}[tdplot_main_coords, scale=3.5]

\def\R{1}
\def\D{1.66}

\coordinate (p3) at (-0.926, 0.3778, 0); 
\coordinate (p4) at (0, -1, 0);
\coordinate (p3m)  at (0.926, 0.3778, 0);
\coordinate(z) at (0,0,0);

\shade[ball color=blue!10, opacity=0.2] (0,0,0) circle (\R);

 \tkzDefPointBy[reflection = over p3--z](p4)
 \tkzGetPoint{p4m}

  \tkzDefMidPoint(p3,p3m)
 \tkzGetPoint{m3}

  \tkzDefMidPoint(p4,p4m)
 \tkzGetPoint{m4}
\path[name path=e1](m3) ellipse[x radius=0.926, y radius=0.15];

\path[name path=e2, rotate around={68:(m4)}] (m4) ellipse[x radius=0.926, y radius=0.15];
 
\path[name intersections={of=e1 and e2, by={i1,i2,i3,i4}}];
\tkzDrawSegment[dotted](p4,p3)
\tkzLabelSegment[orange](p3,p3m){$\Gamma_4$}
\tkzLabelSegment[dgreen, below](p4,p4m){$\Gamma_3$}
\tkzLabelSegment(p4,p3){$D$}
 \tkzInterLL(p3,p3m)(p4,p4m)
 \tkzGetPoint{p1s}

\draw[orange, dashed] ([shift={(-0.926,0)}]m3) arc[start angle=180, end angle=0, x radius=0.926, y radius=0.15];
\draw[orange] ([shift={(-0.926,0)}]m3) arc[start angle=180, end angle=360, x radius=0.926, y radius=0.15];
\draw[dgreen, dashed, rotate around={68:(m4)}] ([shift={(-0.926,0)}]m4)arc[start angle=180, end angle=360, x radius=0.926, y radius=0.15];
\draw[dgreen,rotate around={68:(m4)}] ([shift={(0.926,0)}]m4)arc[start angle=0, end angle=180, x radius=0.926, y radius=0.15];

\filldraw (p3) circle (0.5pt) node[above left] {$p^3$};
\filldraw(p4) circle (0.5pt) node[below] {$p^4$};    
\filldraw(i2) circle (0.5pt) node[above right] {$p^2_*$}; 

\filldraw(i3) circle (0.5pt) node[below right] {$p^1_*$}; 
\end{tikzpicture}

    \caption{The basic configuration: $p^3$ and $p^4$ are fixed to have distance $D$ and the small circles $\Gamma_3$ and $\Gamma_4$ contain the points of the sphere having distance $D$ to $p^3$ and $p^4$, respectively. 
    }
    \label{fig:eucl-basic}
\end{figure}

Now we consider the remaining two vertices $p^1$ and $p^2$. By \Cref{rem:opteuclidean}, we may assume $p^2_3<0$ and $p^1_3>0$. Otherwise, $D$ is at least the diameter of a planar set, contradicting $D<\sqrt{3}$. Thus, $p^2\in \set{x\in\S: \norm{x-p^3}\leq D, \norm{x-p^4}\leq D, x_3<0}$ and $p^1\in \set{x\in\S: \norm{x-p^2}\leq D,\norm{x-p^3}\leq D, \norm{x-p^4}\leq D, x_3>0}$. 
Additionaly, since the simplex needs to be optimally contained in $\B$, 
we also have $p^1\in \set{x\in \S:x\in \pos(\set{-p^2,-p^3,-p^4})}$ by \Cref{prop:opt}. The following lemma describes the topology of the spherical region in which $p^1$ can be located.  

\begin{lemma}
\label{lem:defT}
Let $D\in\left[\sqrt{\frac{8}{3}},\sqrt{3}\right)$ and $S:=\conv\left(\set{p^1,p^2,p^3,p^4}\right)\optc \B$ such that $D = D(S)=\norm{p^3-p^4}$. Moreover, define 
\begin{align*}
        T_1:&=\set{x\in \S: \norm{x-p^2},\norm{x-p^3},\norm{x-p^4}\leq D} \text{ and}\\
        T_2:&=\set{x\in \R^3:x\in \pos(\set{-p^2,-p^3,-p^4})}.
    \end{align*}
Then, $T:=T_1\cap T_2$
    is a simply connected subset of the sphere bounded by three small circles.
\end{lemma}
 Given the points $p^2, p^3, p^4$, the set $T_1$ describes the region of the sphere in which $x$ may be situated, such that $D(\conv(\set{x,p^2,p^3,p^4}))= D$. 
 Additionally, by \Cref{rem:sqrt3} we need $x\in T_2$ in order to have $\conv(\set{x,p^2,p^3,p^4})\optc \B$. 

One should recognize, that $T_1$ for itself might not be simply connected. 
\begin{proof}

Since $D(S)=D $, all points $p^1,p^2,p^3,p^4$ belong to $T_1$ and $p^1,p^2,p^3,p^4\in \S$ because $D < \sqrt{3}$, as explained in \Cref{rem:sqrt3}. Additionally, since $S\optc \B$, at least $p^1\in T_2$ and therefore $T\neq \emptyset$.

 
 W.l.o.g., we may assume that $p^3$ and $p^4$ are defined as in \eqref{eq:p3p4def} and that $p^2_3<0$. 

 Moreover, $ T_1 \cap \bd( T_2)  \neq \emptyset$ would imply that we may choose $x \in T$ such that the convex hull of three points out of $x,p^2,p^3,p^4$ is optimally contained in $\B$. In this case, the diameter of the according triangle would already be at least $\sqrt{3}$, contradicting $D(S) < \sqrt{3}$. Thus, $T_1\cap \bd (T_2)=\emptyset$. 
 Now, recognize that, by our assumptions on $p^2,p^3,p^4$ we have $x_3> 0$ for every $x\in \inte (T_2)$.

 For short, we say that two points in $T_1$ are \textit{connected} if they can be connected by a path within $T_1$. Thus, $T_1\cap \bd(T_2)=\emptyset$ implies that no point in $T_1\cap T_2$ can be connected with a point in $T_1$ with a negative third coordinate. We call this property (P1). 


Using the notation before the lemma, $\Gamma_3$ and $\Gamma_4$ intersect in the points $p^1_{*}$ and $p^2_{*}$. Moreover, by \Cref{rem:p3p4} 
both circles intersect $\Gamma_2$ in two points. In the following, we consider the cases of how the circles can intersect and show that all but the last one will lead to a contradiction.

  The two intersection points divide both circles, $\Gamma_3$ and $\Gamma_4$, into two parts, one with distance to $p^2$ at most $D$ and one with a larger distance than $D$. We distinguish between the cases $p^1_*\notin T_1$ and $p^1_* \in T_1$. 

\textit{Case 1:} First, assume $p^1_* \notin T_1$ (\cf~\Cref{fig:case21}), \ie~$\norm{p^1_*-p^2}>D$. 

\textit{Case 1.1:}
If points of both small circles, $\Gamma_3$ and $\Gamma_4$, belong to $T_1$ with negative third coordinate, every point in $T_1$ is connected to a point in $T_1$ with a negative third coordinate. By (P1) we obtain $T_1\cap T_2=\emptyset$,  contradicting $p^1\in T_1\cap T_2$.

\textit{Case 1.2:} Now, assume that neither $\Gamma_3\cap T_1$ nor $\Gamma_4\cap T_1$ contain points with a negative third coordinate. 
Since $p^3,p^4\in T_1$ and $p^3_3=p^4_3=0$ the points $p^3$ and $p^4$ need to be endpoints of $\Gamma_4\cap T_1$ and $\Gamma_3 \cap T_1$, respectively, which means they are the intersection points of $\Gamma_2$ with $\Gamma_4$ and $\Gamma_3$, respectively. Together, this would imply $p^2=p^2_*$ and because of $\norm{p^1_*-p^2_*}=D\sqrt{4-D^2-\frac{(D^2-2)^2}{4-D^2}}\leq D$ therefore $p^1_*\in T_1$, contradicting our assumption. 

\textit{Case 1.3:} 
Completing Case 1, consider the case that there belong points to one of the two small circles with negative third coordinate, but not both. W.l.o.g let $\Gamma_4$ be the one with such points. 
Then (as argued for Case 1.2), 
$p^4\in \Gamma_2$ is an endpoint of the arc of points in $\Gamma_3$ with distance at most $D$ to $p^2$.  
Moreover, $p^4\in \Gamma_2$ also implies 
$p^2\in \Gamma_4$. 
Since $p^2_3<0$, $p^2$ can only belong to the arc of $\Gamma_4$ between $p^3$ and $p^2_*$. In that case $\norm{p^2_*-p^2}<D$, which would imply $p^2_*\in T_1$.
This configuration cannot be achieved since all points belonging to the arc of $\Gamma_3$ within $T_1$ should only have non-negative third coordinate by our assumption. However, $p^2_*\in T_1\cap \Gamma_3$ and we have $(p^2_*)_3<0$. 


  \begin{figure}
\begin{minipage}{0.4\linewidth}
        \centering
    \tdplotsetmaincoords{0}{0}
\begin{tikzpicture}[tdplot_main_coords, scale=2.2]

\def\R{1}
\def\D{1.66}

\coordinate (p3) at (-0.926, 0.3778, 0); 
\coordinate (p4) at (0, -1, 0);
\coordinate (p3m)  at (0.926, 0.3778, 0);
\coordinate(z) at (0,0,0);

\shade[ball color=blue!10, opacity=0.2] (0,0,0) circle (\R);

 \tkzDefPointBy[reflection = over p3--z](p4)
 \tkzGetPoint{p4m}

  \tkzDefMidPoint(p3,p3m)
 \tkzGetPoint{m3}

  \tkzDefMidPoint(p4,p4m)
 \tkzGetPoint{m4}
\path[name path=e1](m3) ellipse[x radius=0.926, y radius=0.15];

\path[name path=e2, rotate around={68:(m4)}] (m4) ellipse[x radius=0.926, y radius=0.15];
 
\path[name intersections={of=e1 and e2, by={i1,i2,i3,i4}}];
\tkzLabelSegment[orange,above=1.5ex](p3,p3m){$\Gamma_4$}
\tkzLabelSegment[dgreen, below](p4,p4m){$\Gamma_3$}
 \tkzInterLL(p3,p3m)(p4,p4m)
 \tkzGetPoint{p1s}

\draw[orange, dashed] ([shift={(-0.926,0)}]m3) arc[start angle=180, end angle=0, x radius=0.926, y radius=0.15];
\draw[orange] ([shift={(-0.926,0)}]m3) arc[start angle=180, end angle=360, x radius=0.926, y radius=0.15];
\draw[dgreen, dashed, rotate around={68:(m4)}] ([shift={(-0.926,0)}]m4)arc[start angle=180, end angle=360, x radius=0.926, y radius=0.15];
\draw[dgreen,rotate around={68:(m4)}] ([shift={(0.926,0)}]m4)arc[start angle=0, end angle=180, x radius=0.926, y radius=0.15];

\def\xa{0.926}
\def\ya{0.15}

\def\startAngle{220}
\def\endAngle{180}

\path (m3) ++({\xa*cos(\startAngle)}, {\ya*sin(\startAngle)}) coordinate (A);
\path (m3) ++({\xa*cos(\endAngle)}, {\ya*sin(\endAngle)}) coordinate (B);
\draw[red, thick] (A) arc[start angle=\startAngle, end angle=\endAngle, x radius=\xa, y radius=\ya];

\def\theta{68}

\def\arcStartCD{140}
\def\arcEndCD{180}

\path (m4) ++({\xa*cos(\arcStartCD)}, {\ya*sin(\arcStartCD)}) coordinate (preC);
\path (m4) ++({\xa*cos(\arcEndCD)}, {\ya*sin(\arcEndCD)}) coordinate (preD);

\begin{scope}[rotate around={\theta:(m4)}]
  \draw[red, thick]
    ($(m4) + ({\xa*cos(\arcStartCD)}, {\ya*sin(\arcStartCD)})$)
    arc[start angle=\arcStartCD, end angle=\arcEndCD, x radius=\xa, y radius=\ya];
\end{scope}
\tkzDefPointBy[rotation=center m4 angle \theta](preC)
\tkzGetPoint{C}
\tkzDefPointBy[rotation=center m4 angle \theta](preD)
\tkzGetPoint{D}

\filldraw (A) circle (0.5pt);
\filldraw (B) circle (0.5pt);
\filldraw (C) circle (0.5pt);
\filldraw (D) circle (0.5pt);
\filldraw (p3) circle (0.5pt) node[above left] {$p^3$};
\filldraw(p4) circle (0.5pt) node[below] {$p^4$};    
\filldraw(i2) circle (0.5pt) node[above right] {$p^2_*$}; 

\filldraw(i3) circle (0.5pt) node[below] {$p^1_*$}; 
\end{tikzpicture}
\end{minipage}
\begin{minipage}{0.4\linewidth}
        \centering
    \tdplotsetmaincoords{0}{0}
\begin{tikzpicture}[tdplot_main_coords, scale=2.2]

\def\R{1}
\def\D{1.66}

\coordinate (p3) at (-0.926, 0.3778, 0); 
\coordinate (p4) at (0, -1, 0);
\coordinate (p3m)  at (0.926, 0.3778, 0);
\coordinate(z) at (0,0,0);

\shade[ball color=blue!10, opacity=0.2] (0,0,0) circle (\R);

 \tkzDefPointBy[reflection = over p3--z](p4)
 \tkzGetPoint{p4m}

  \tkzDefMidPoint(p3,p3m)
 \tkzGetPoint{m3}

  \tkzDefMidPoint(p4,p4m)
 \tkzGetPoint{m4}
\path[name path=e1](m3) ellipse[x radius=0.926, y radius=0.15];

\path[name path=e2, rotate around={68:(m4)}] (m4) ellipse[x radius=0.926, y radius=0.15];
 
\path[name intersections={of=e1 and e2, by={i1,i2,i3,i4}}];
\tkzLabelSegment[orange,above=1.5ex](p3,p3m){$\Gamma_4$}
\tkzLabelSegment[dgreen, below](p4,p4m){$\Gamma_3$}
 \tkzInterLL(p3,p3m)(p4,p4m)
 \tkzGetPoint{p1s}

\draw[orange, dashed] ([shift={(-0.926,0)}]m3) arc[start angle=180, end angle=0, x radius=0.926, y radius=0.15];
\draw[orange] ([shift={(-0.926,0)}]m3) arc[start angle=180, end angle=360, x radius=0.926, y radius=0.15];
\draw[dgreen, dashed, rotate around={68:(m4)}] ([shift={(-0.926,0)}]m4)arc[start angle=180, end angle=360, x radius=0.926, y radius=0.15];
\draw[dgreen,rotate around={68:(m4)}] ([shift={(0.926,0)}]m4)arc[start angle=0, end angle=180, x radius=0.926, y radius=0.15];

\def\xa{0.926}
\def\ya{0.15}

\def\startAngle{260}
\def\endAngle{150}

\path (m3) ++({\xa*cos(\startAngle)}, {\ya*sin(\startAngle)}) coordinate (A);
\path (m3) ++({\xa*cos(\endAngle)}, {\ya*sin(\endAngle)}) coordinate (B);
\draw[red, thick] (A) arc[start angle=\startAngle, end angle=\endAngle, x radius=\xa, y radius=\ya];

\def\theta{68}

\def\arcStartCD{100}
\def\arcEndCD{180}

\path (m4) ++({\xa*cos(\arcStartCD)}, {\ya*sin(\arcStartCD)}) coordinate (preC);
\path (m4) ++({\xa*cos(\arcEndCD)}, {\ya*sin(\arcEndCD)}) coordinate (preD);

\begin{scope}[rotate around={\theta:(m4)}]
  \draw[red, thick]
    ($(m4) + ({\xa*cos(\arcStartCD)}, {\ya*sin(\arcStartCD)})$)
    arc[start angle=\arcStartCD, end angle=\arcEndCD, x radius=\xa, y radius=\ya];
\end{scope}
\tkzDefPointBy[rotation=center m4 angle \theta](preC)
\tkzGetPoint{C}
\tkzDefPointBy[rotation=center m4 angle \theta](preD)
\tkzGetPoint{D}

\filldraw (A) circle (0.5pt);
\filldraw (B) circle (0.5pt);
\filldraw (C) circle (0.5pt);
\filldraw (D) circle (0.5pt);
\filldraw (p3) circle (0.5pt) node[above left] {$p^3$};
\filldraw(p4) circle (0.5pt) node[below] {$p^4$};    
\filldraw(i2) circle (0.5pt) node[above right] {$p^2_*$}; 

\filldraw(i3) circle (0.5pt) node[below] {$p^1_*$}; 
\end{tikzpicture}
\end{minipage}

    \caption{Case 1: If $p^1_* \notin T_1$ then $T_1\cap T_2 = \emptyset$. The points on $\Gamma_3$ and $\Gamma_4$ that have distance at most $D$ to $p^2$ are marked in red. 
    Case 1.2 (left): If on both small circles, $\Gamma_3$ and $\Gamma_4$, there are no points in $T_1$ with a negative third coordinate, $p^3$ as well as $p^4$ need to be intersection points, implying $p^2=p^2_*$. Case 1.3 (right): If this only is the case for $\Gamma_4$, we still have $p^2\in \Gamma_4$.
    }
    \label{fig:case21}
\end{figure}
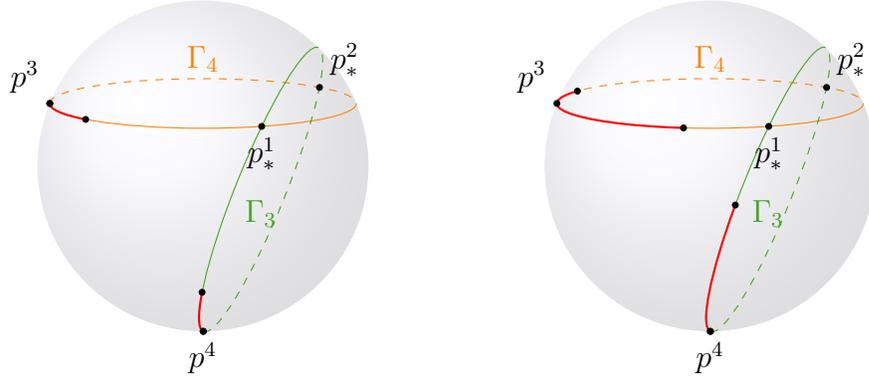

  \textit{Case 2:} Since the assumption $p^1_*\notin T_1$ led to a contradiction in all subcases, we see that $p^1_*\in T_1$ must be true (\cf~\Cref{fig:case22}). 
  As mentioned before, $p^3$ and $p^4$ are also contained in $T_1$. Thus, the two intersection points of $\Gamma_2$ and $\Gamma_4$ may be both in the part of $\Gamma_4$ between $p^1_*$ and $p^3$  with no negative third coordinates ("front") or both in the other part ("back"). The same holds true for the pair of intersection points of $\Gamma_2\cap \Gamma_3$, now with the front / back parts of $\Gamma_3$ between $p^1_*$ and $p^4$. 

  \textit{Case 2.1:} 
 Assume both pairs are in the back. By a similar argument as in Case 1.3, this configuration is not possible if $p^2=p^2_*$. 
Thus, let $p^2\neq p^2_*$. Then, all points in $T_1$ are connected within $T_1$ to a point with negative third coordinate on $\Gamma_3$ or $\Gamma_4$: since every point in $T_1$ needs to be connected to at least one of the circles $\Gamma_3$ or $\Gamma_4$, and on these circles, close to $p^3$ or $p^4$, there exist points with negative third coordinate. 
Thus, in this case (P1) would imply $T_1\cap T_2=\emptyset$,  contradicting $p^1 \in T_1 \cap T_2$.
 
 \textit{Case 2.2: }Now, assume the pairs are in different parts, w.l.o.g. in the back for $\Gamma_3$ and in the front for $\Gamma_4$. The front pair implies $p^2_*\in T_1$. Furthermore, $p^4\in \Gamma_2\cap \Gamma_3$, since otherwise all points in $T_1$ are connected to one with negative third coordinate, which would imply $T_1\cap T_2=\emptyset$ by (P1). 
 Thus, the second intersection point on $\Gamma_3$  needs to have a negative third coordinate. Now, $p^4\in \Gamma_2$ implies $p^2\in \Gamma_4$ and therefore, $p^2=(\alpha, \frac{D^2}{2}-1, \gamma)^{\top}$ for some $\alpha\geq p^3_1=-\sqrt{D^2-\frac{D^4}{4}}$ and $\gamma<0$. Using \eqref{eq:distancehalf-space} we obtain
\begin{equation*}
\begin{split}
    \Gamma_2\cap \Gamma_3=\left\{x\in\S: \alpha x_1 +\left(\frac{D^2}{2}-1\right)x_2+\gamma x_3\right.&=- \left(\frac{D^2}{2}-1\right),\\
   -\sqrt{D^2-\frac{D^4}{4}} x_1 +\left(\frac{D^2}{2}-1\right)x_2&= \left.- \left(\frac{D^2}{2}-1\right) \right\}.
\end{split}
\end{equation*}
Subtracting the two equations yields
\begin{equation*}
     \left(\alpha+\sqrt{D^2-\frac{D^4}{4}}\right) x_1+ \gamma x_3=0 \quad 
     \Longleftrightarrow \quad 
     x_3=-\frac{1}{\gamma} \left(\alpha+\sqrt{D^2-\frac{D^4}{4}}\right) x_1.
\end{equation*}
Since $x_1\geq 0$ for all $x\in \Gamma_3$, we have $x_3\geq 0$, too, which contradicts our previous conclusion that the intersection points of $\Gamma_2$ and  $\Gamma_3$ are $p^4$ and a point with negative third coordinate. 

 \textit{Case 2.3: }Finally, if both pairs are in the front, $T_1$ has two components. Since we are intersecting $\Gamma_3$ and $\Gamma_4$ also with the small circle $\Gamma_2$, we obtain a part containing $p^1_*$ and a part containing $p^3$ and $p^4$. The second one contains points with negative third coordinates and does therefore not intersect $T_2$ by (P1). The remaining first component cannot be empty as it must contain $p^1$.  
 
 Since Case 2.3 is the only case that does not lead to a contradiction, $T=T_1\cap T_2$ is always exactly this one component, a simply connected set bounded by $\Gamma_2$, $\Gamma_3$, and $\Gamma_4$ and containing $p^1_*$. 
 

  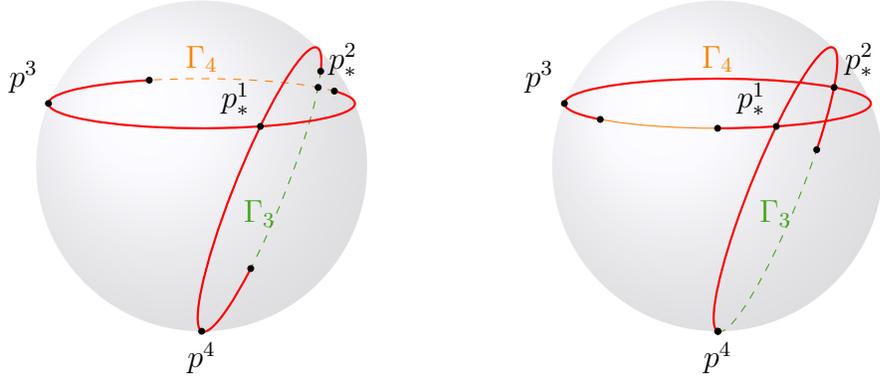
\begin{figure}
\begin{minipage}{0.4\linewidth}
        \centering
    \tdplotsetmaincoords{0}{0}
\begin{tikzpicture}[tdplot_main_coords, scale=2.2]

\def\R{1}
\def\D{1.66}

\coordinate (p3) at (-0.926, 0.3778, 0); 
\coordinate (p4) at (0, -1, 0);
\coordinate (p3m)  at (0.926, 0.3778, 0);
\coordinate(z) at (0,0,0);

\shade[ball color=blue!10, opacity=0.2] (0,0,0) circle (\R);

 \tkzDefPointBy[reflection = over p3--z](p4)
 \tkzGetPoint{p4m}

  \tkzDefMidPoint(p3,p3m)
 \tkzGetPoint{m3}

  \tkzDefMidPoint(p4,p4m)
 \tkzGetPoint{m4}
\path[name path=e1](m3) ellipse[x radius=0.926, y radius=0.15];

\path[name path=e2, rotate around={68:(m4)}] (m4) ellipse[x radius=0.926, y radius=0.15];
 
\path[name intersections={of=e1 and e2, by={i1,i2,i3,i4}}];
\tkzLabelSegment[orange,above=1.5ex](p3,p3m){$\Gamma_4$}
\tkzLabelSegment[dgreen, below](p4,p4m){$\Gamma_3$}
 \tkzInterLL(p3,p3m)(p4,p4m)
 \tkzGetPoint{p1s}

\draw[orange, dashed] ([shift={(-0.926,0)}]m3) arc[start angle=180, end angle=0, x radius=0.926, y radius=0.15];
\draw[orange] ([shift={(-0.926,0)}]m3) arc[start angle=180, end angle=360, x radius=0.926, y radius=0.15];
\draw[dgreen, dashed, rotate around={68:(m4)}] ([shift={(-0.926,0)}]m4)arc[start angle=180, end angle=360, x radius=0.926, y radius=0.15];
\draw[dgreen,rotate around={68:(m4)}] ([shift={(0.926,0)}]m4)arc[start angle=0, end angle=180, x radius=0.926, y radius=0.15];

\def\xa{0.926}
\def\ya{0.15}

\def\startAngle{390}
\def\endAngle{110}

\path (m3) ++({\xa*cos(\startAngle)}, {\ya*sin(\startAngle)}) coordinate (A);
\path (m3) ++({\xa*cos(\endAngle)}, {\ya*sin(\endAngle)}) coordinate (B);
\draw[red, thick] (A) arc[start angle=\startAngle, end angle=\endAngle, x radius=\xa, y radius=\ya];

\def\theta{68}

\def\arcStartCD{330}
\def\arcEndCD{600}

\path (m4) ++({\xa*cos(\arcStartCD)}, {\ya*sin(\arcStartCD)}) coordinate (preC);
\path (m4) ++({\xa*cos(\arcEndCD)}, {\ya*sin(\arcEndCD)}) coordinate (preD);

\begin{scope}[rotate around={\theta:(m4)}]
  \draw[red, thick]
    ($(m4) + ({\xa*cos(\arcStartCD)}, {\ya*sin(\arcStartCD)})$)
    arc[start angle=\arcStartCD, end angle=\arcEndCD, x radius=\xa, y radius=\ya];
\end{scope}
\tkzDefPointBy[rotation=center m4 angle \theta](preC)
\tkzGetPoint{C}
\tkzDefPointBy[rotation=center m4 angle \theta](preD)
\tkzGetPoint{D}

\filldraw (A) circle (0.5pt);
\filldraw (B) circle (0.5pt);
\filldraw (C) circle (0.5pt);
\filldraw (D) circle (0.5pt);
\filldraw (p3) circle (0.5pt) node[above left] {$p^3$};
\filldraw(p4) circle (0.5pt) node[below] {$p^4$};    
\filldraw(i2) circle (0.5pt) node[above right] {$p^2_*$}; 

\filldraw(i3) circle (0.5pt) node[above left] {$p^1_*$}; 
\end{tikzpicture}
\end{minipage}
\begin{minipage}{0.4\linewidth}
        \centering
    \tdplotsetmaincoords{0}{0}
\begin{tikzpicture}[tdplot_main_coords, scale=2.2]

\def\R{1}
\def\D{1.66}

\coordinate (p3) at (-0.926, 0.3778, 0); 
\coordinate (p4) at (0, -1, 0);
\coordinate (p3m)  at (0.926, 0.3778, 0);
\coordinate(z) at (0,0,0);

\shade[ball color=blue!10, opacity=0.2] (0,0,0) circle (\R);

 \tkzDefPointBy[reflection = over p3--z](p4)
 \tkzGetPoint{p4m}

  \tkzDefMidPoint(p3,p3m)
 \tkzGetPoint{m3}

  \tkzDefMidPoint(p4,p4m)
 \tkzGetPoint{m4}
\path[name path=e1](m3) ellipse[x radius=0.926, y radius=0.15];

\path[name path=e2, rotate around={68:(m4)}] (m4) ellipse[x radius=0.926, y radius=0.15];
 
\path[name intersections={of=e1 and e2, by={i1,i2,i3,i4}}];
\tkzLabelSegment[orange,above=1.5ex](p3,p3m){$\Gamma_4$}
\tkzLabelSegment[dgreen, below](p4,p4m){$\Gamma_3$}
 \tkzInterLL(p3,p3m)(p4,p4m)
 \tkzGetPoint{p1s}

\draw[orange, dashed] ([shift={(-0.926,0)}]m3) arc[start angle=180, end angle=0, x radius=0.926, y radius=0.15];
\draw[orange] ([shift={(-0.926,0)}]m3) arc[start angle=180, end angle=360, x radius=0.926, y radius=0.15];
\draw[dgreen, dashed, rotate around={68:(m4)}] ([shift={(-0.926,0)}]m4)arc[start angle=180, end angle=360, x radius=0.926, y radius=0.15];
\draw[dgreen,rotate around={68:(m4)}] ([shift={(0.926,0)}]m4)arc[start angle=0, end angle=180, x radius=0.926, y radius=0.15];

\def\xa{0.926}
\def\ya{0.15}

\def\startAngle{220}
\def\endAngle{-90}

\path (m3) ++({\xa*cos(\startAngle)}, {\ya*sin(\startAngle)}) coordinate (A);
\path (m3) ++({\xa*cos(\endAngle)}, {\ya*sin(\endAngle)}) coordinate (B);
\draw[red, thick] (A) arc[start angle=\startAngle, end angle=\endAngle, x radius=\xa, y radius=\ya];

\def\theta{68}

\def\arcStartCD{-70}
\def\arcEndCD{180}

\path (m4) ++({\xa*cos(\arcStartCD)}, {\ya*sin(\arcStartCD)}) coordinate (preC);
\path (m4) ++({\xa*cos(\arcEndCD)}, {\ya*sin(\arcEndCD)}) coordinate (preD);

\begin{scope}[rotate around={\theta:(m4)}]
  \draw[red, thick]
    ($(m4) + ({\xa*cos(\arcStartCD)}, {\ya*sin(\arcStartCD)})$)
    arc[start angle=\arcStartCD, end angle=\arcEndCD, x radius=\xa, y radius=\ya];
\end{scope}
\tkzDefPointBy[rotation=center m4 angle \theta](preC)
\tkzGetPoint{C}
\tkzDefPointBy[rotation=center m4 angle \theta](preD)
\tkzGetPoint{D}

\filldraw (A) circle (0.5pt);
\filldraw (B) circle (0.5pt);
\filldraw (C) circle (0.5pt);
\filldraw (D) circle (0.5pt);
\filldraw (p3) circle (0.5pt) node[above left] {$p^3$};
\filldraw(p4) circle (0.5pt) node[below] {$p^4$};    
\filldraw(i2) circle (0.5pt) node[above right] {$p^2_*$}; 

\filldraw(i3) circle (0.5pt) node[above left] {$p^1_*$}; 
\end{tikzpicture}
\end{minipage}

    \caption{Case 2: If $p^1_* \in T_1$ and not both intersection pairs are in front, $T_1\cap T_2$ is empty. 
    The points on $\Gamma_3$ and $\Gamma_4$ that have distance at most $D$ to $p^2$ are marked in red.
    Case 2.1 (left): If both pairs of intersection pairs are in the back and $p^2\neq p^2_*$, all points in $T_1$ are connected to a point with
negative third coordinate. Case 2.2  (right): If one is in the back and one in front, we can assume that $p^4 \in \Gamma_2\cap \Gamma_3$ and obtain a contradiction. }
    \label{fig:case22}
\end{figure}

\end{proof}

If 
$T \neq \set{p^1_*}$, then it is a triangle-like shape defined by three small circles with three vertices (\cf~\Cref{fig:areaT}). In the following, we denote 
the two vertices besides $p^1_*$ by $t^3$, $t^4$, where
$t^3\in T_2 \cap \Gamma_2 \cap \Gamma_3$ and $t^4 \in T_2 \cap \Gamma_2 \cap \Gamma_4$. 
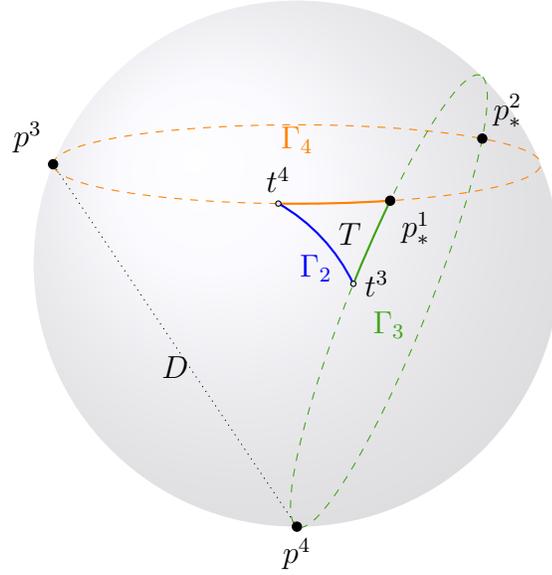
\begin{figure}[ht]
    \centering
    \tdplotsetmaincoords{0}{0}
\begin{tikzpicture}[tdplot_main_coords, scale=3.5]

\def\R{1}
\def\D{1.66}

\coordinate (p3) at (-0.926, 0.3778, 0); 
\coordinate (p4) at (0, -1, 0);
\coordinate (p3m)  at (0.926, 0.3778, 0);
\coordinate(z) at (0,0,0);

\shade[ball color=blue!10, opacity=0.2] (0,0,0) circle (\R);

 \tkzDefPointBy[reflection = over p3--z](p4)
 \tkzGetPoint{p4m}

  \tkzDefMidPoint(p3,p3m)
 \tkzGetPoint{m3}

  \tkzDefMidPoint(p4,p4m)
 \tkzGetPoint{m4}
\path[name path=e1](m3) ellipse[x radius=0.926, y radius=0.15];

\path[name path=e2, rotate around={68:(m4)}] (m4) ellipse[x radius=0.926, y radius=0.15];
 
\path[name intersections={of=e1 and e2, by={i1,i2,i3,i4}}];
\tkzDrawSegment[dotted](p4,p3)
\tkzLabelSegment[orange, above](p3,p3m){$\Gamma_4$}
\tkzLabelSegment[dgreen, below](p4,p4m){$\Gamma_3$}
\tkzLabelSegment(p4,p3){$D$}
 \tkzInterLL(p3,p3m)(p4,p4m)
 \tkzGetPoint{p1s}

\draw[orange, dashed] ([shift={(-0.926,0)}]m3) arc[start angle=180, end angle=-180, x radius=0.926, y radius=0.15];


\draw[dgreen,dashed, rotate around={68:(m4)}] ([shift={(0.926,0)}]m4)arc[start angle=0, end angle=360, x radius=0.926, y radius=0.15];

\def\xa{0.926}
\def\ya{0.15}

\def\startAngle{265}
\def\endAngle{292}

\path (m3) ++({\xa*cos(\startAngle)}, {\ya*sin(\startAngle)}) coordinate (A);
\path (m3) ++({\xa*cos(\endAngle)}, {\ya*sin(\endAngle)}) coordinate (B);
\draw[orange, thick] (A) arc[start angle=\startAngle, end angle=\endAngle, x radius=\xa, y radius=\ya];

\def\theta{68}

\def\arcStartCD{68}
\def\arcEndCD{90}

\path (m4) ++({\xa*cos(\arcStartCD)}, {\ya*sin(\arcStartCD)}) coordinate (preC);
\path (m4) ++({\xa*cos(\arcEndCD)}, {\ya*sin(\arcEndCD)}) coordinate (preD);

\begin{scope}[rotate around={\theta:(m4)}]
  \draw[dgreen, thick]
    ($(m4) + ({\xa*cos(\arcStartCD)}, {\ya*sin(\arcStartCD)})$)
    arc[start angle=\arcStartCD, end angle=\arcEndCD, x radius=\xa, y radius=\ya];
\end{scope}
\tkzDefPointBy[rotation=center m4 angle \theta](preC)
\tkzGetPoint{C}
\tkzDefPointBy[rotation=center m4 angle \theta](preD)
\tkzGetPoint{D}

\tkzDefPointOnLine[pos=0.45](p4,p3)
 \tkzGetPoint{m5}
\path[name path=e3, rotate around={292:(m5)}] (m5) ellipse[x radius=0.7, y radius=0.7];

\path[name intersections={of=e3 and e2, by={k1,k2,k3,k4}}];
\path[name intersections={of=e3 and e1, by={h1,h2,h3,h4}}];

\tkzDrawArc[ thick, dblue](m5,k2)(h2)


\filldraw (p3) circle (0.5pt) node[above left] {$p^3$};
\filldraw(p4) circle (0.5pt) node[below] {$p^4$};    
\filldraw(i2) circle (0.5pt) node[above right] {$p^2_*$}; 

\filldraw(i3) circle (0.5pt) node[below right] {$p^1_*$}; 

\tkzLabelSegment[pos=0.65, above right](h2,k2){$T$}
\tkzLabelSegment[pos=0.5, blue](k2,h2){$\Gamma_2$}
\tkzLabelPoint[above](h2){$t^4$}
\tkzLabelPoint[right](k2){$t^3$}
\tkzDrawPoints(k2,h2)
\end{tikzpicture}

    \caption{If $p^2$ is chosen such that $T$ is not empty, it is a triangular region bounded by $\Gamma_2$, $\Gamma_3$, and $\Gamma_4$. }
    \label{fig:areaT}
\end{figure}

\begin{lemma}
\label{lem:choicesp1}
Let $p^2,p^3,p^4,p^1_{*}$ and $T$ be defined as in \Cref{lem:defT} and $t^3,t^4$ as above. Then, 
\begin{equation*}
    T \subset \conv_{C,p^2}\left(\set{t^3,t^4,p^1_*}\right).
\end{equation*}


\end{lemma}
\begin{proof} 
The proof works as follows: we show that the $C,p^2$-convex hulls of each pair of the vertices of $T$ does not intersect the relative interior of $T$, implying our claim $T\subset \conv_{C,p^2}\left(\set{t^3,t^4,p^1_*}\right)$.

For the pair $t^3,t^4$ we 
show that if 
$p \in \conv_{C,p^2}(\set{t^3,t^4})$, then $\norm{p-p^2}\geq D$, which implies 
$\conv_{C,p^2}(\set{t^3,t^4})$ does not intersect the interior of $T$.

By definition, $p \in \conv_{C,p^2}(\set{t^3,t^4})$ means that there exist $\alpha,\beta\geq 0$ with $\alpha+\beta\geq 1$ such that 
\begin{equation*}
    p=p^2+\alpha (t^4-p^2)+\beta(t^3-p^2). 
\end{equation*}
We know by \eqref{eq:distancehalf-space} that $-(p^2)^{\top}t^4=-(p^2)^{\top}t^3=\frac{D^2}{2}-1$ and that $-(p^2)^{\top}x \geq\frac{D^2}{2}-1$ is equivalent to $\norm{x-p^2}\geq D$. However,
\begin{align*}
    -(p^2)^{\top}p &= -(p^2)^{\top}p^2 - \alpha((p^2)^{\top}t^4-(p^2)^{\top}p^2)-\beta((p^2)^{\top}t^3-(p^2)^{\top}p^2)\\
    &= -1 +\alpha \left(\frac{D^2}{2}-1\right) +\alpha +\beta \left( \frac{D^2}{2}-1\right) +\beta\geq \frac{D^2}{2}-1,
\end{align*} and therefore $\norm{p-p^2}\geq D$. 

For the pairs $p^1_*,t^i$, $i\in\{3,4\}$, $p \in \conv_{C,p^2}(\set{p^1_*,t^i})$ implies there exist $\alpha,\beta\geq 0$ with $\alpha+\beta\geq 1$ with
    \begin{equation*}
        p=p^2+\alpha (t^i-p^2)+\beta(p^1_*-p^2).
    \end{equation*}
Using $\norm{p^i-p^2}\leq D$, we obtain
\begin{align*}
    -(p^i)^{\top}p &= -(p^i)^{\top}p^i - \alpha((p^i)^{\top}t^i-(p^i)^{\top}p^2)-\beta((p^i)^{\top}p^1_*-(p^i)^{\top}p^2)\\
    &=(-1+\alpha+\beta)(p^i)^{\top}p^2 +\alpha \left( \frac{D^2}{2}-1\right) +\beta \left( \frac{D^2}{2}-1\right) \\
   &\geq (-1+\alpha+\beta)\left(- \frac{D^2}{2}+1\right)+\alpha \left( \frac{D^2}{2}-1\right)+\beta\left( \frac{D^2}{2}-1\right) =  \frac{D^2}{2}-1,
\end{align*} 
 and therefore $\norm{p-p^i}\geq D$. 
\end{proof}

Since we have $\set{t^3,t^4,p^1_*}\subset T \subset\conv_{C,p^2}\left(\set{t^3,t^4,p^1_*}\right) $, the above lemma, together with \Cref{cor:bdarea}, proves that the smallest inradius of $S$ is attained 
for $p^1 \in \set{p^1_*,t^3,t^4}$. 

Next, we prove that it suffices to consider simplices $S$ with the property that each vertex of $S$ is adjacent to at least two edges of length $D$ or, in other words, simplices with at least four diametrical edges and at most two opposing edges of shorter length. 

\begin{lemma}
    \label{lem:optinradius1}
    The inradius of every three-dimensional simplex $S$ optimally contained in $\B$ with $D(S)\in\left[\sqrt{\frac{8}{3}}, \sqrt{3}\right)$ is at least the inradius of a simplex optimally contained in $\B$ with the property that each vertex is adjacent to at least two edges of length $D(S)$.
    
\end{lemma}
\begin{proof}
Let $D:=D(S)\in \left[\sqrt{\tfrac{8}{3}}, \sqrt{3}\right)$ and $\set{p^1,p^2,p^3,p^4}$ be the vertices of the simplex. Since  $D\in\left[\sqrt{\tfrac{8}{3}}, \sqrt{3}\right)$, we know that all four vertices belong to $\S$. By rotational symmetry, we may assume that 
the diameter of $S$ is attained between $p^3$ and $p^4$ as well as  
$p^2_3<0$ and $p^1_3>0$. 
Applying \Cref{lem:choicesp1} (in combination with \Cref{cor:bdarea}) with the roles of $p^1$ and $p^2$ swapped the smallest inradius is attained at one of the vertices of 
\begin{equation*}
\begin{split}
    \Tilde{T}:=&\set{x\in\S: \norm{x-p^1}\leq D, \norm{x-p^3}\leq D, \norm{x-p^4}\leq D }\\&\cap \set{x\in \S:x\in \pos(\set{-p^1,-p^3,-p^4})}.
\end{split}
\end{equation*}
Since at each of the three vertices of $\tilde T$ we have $\norm{p^2-p^3}=D$ or $\norm{p^2-p^4}=D$, we may assume 
w.l.o.g.~that  $\norm{p^2-p^3}=D$, too. 

Now, we apply \Cref{lem:choicesp1} (and \Cref{cor:bdarea}) for $p^1$ itself and compare the inradii for the three cases $p^1\in\set{p^1_*,t^3,t^4}$. 

First, let $p^1=p^1_*$. 
In this case, the triangle with vertices $p^1_*,p^3,p^4$ is regular. Thus, because of rotational symmetry (around the axis orthogonal to $\aff(\set{p^1_*,p^3,p^4})$ through $0$), the choice of $p^2$ out of the three vertices of $\tilde{T}$ does not change the inradius of $\conv\left(\set{p^1_*,p^2,p^3,p^4}\right)$ and since $p^2_*$ is a vertex of $\tilde T$, w.l.o.g., we may choose $p^2_*$. Thus, by \Cref{lem:choicesp1} (in combination with \Cref{cor:bdarea}), 
\begin{equation*}
    r(\conv\left(\set{p^1_*,p^2,p^3,p^4}\right)\geq r(\conv\left(\set{p^1_*,p^2_*,p^3,p^4}\right)).
\end{equation*}
and all but one edge of $\conv\left(\set{p^1_*,p^2_*,p^3,p^4}\right)$ have length $D$. 

Second, if $p^1=t^3$ we obtained the same configuration as with $p^1_*$ mirrored with respect to the hyperplane orthogonal to $[p^2,p^4]$ through $p^3$.

Finally, in case of $p^1=t^4$, four edges have length $D$, and only the non-adjacent edges ($[p^2,p^4]$ and $[p^1,p^3]$) may be shorter.
\end{proof}

In the next lemma, we give a general formula for the inradius for simplices with four edges of diametrical length and two opposing edges that could be shorter. 
\begin{lemma}
\label{lem:specialsimplices}
    Let $S=\conv(\set{x^1,x^2,x^3,x^4})\optc R \B + t$ be a simplex with $\set{x^1,x^2,x^3,x^4}\subset R\S +t$. Furthermore, let four edges of $S$ have length $D:=D(S)$ and two opposing edges of lengths $a$ and $b$possibly shorter than $D$.
    Then, 
    \begin{equation*}
        r(S)=\frac{\left(\sqrt{R^2-\frac{a^2}{4}}+\sqrt{R^2-\frac{b^2}{4}}\right)ab}{2a\sqrt{D^2-\frac{a^2}{4}}+2b\sqrt{D^2-\frac{b^2}{4}}}.
    \end{equation*}
\end{lemma}
\begin{proof} Surely, we may assume w.l.o.g.~that $a \le b$ and by translational invariance that $t=0$.
    Moreover, since $\B$ is invariant under rotation, we may also assume
    \begin{equation*}
        x^1=\begin{pmatrix}
            0 \\ \frac{a}{2} \\ \sqrt{R^2-\frac{a^2}{4}}
        \end{pmatrix},\quad x^2=\begin{pmatrix}
             0 \\ -\frac{a}{2} \\ \sqrt{R^2-\frac{a^2}{4}}
        \end{pmatrix}, \quad x^3=\begin{pmatrix}
            \frac{b}{2} \\ 0 \\ -\sqrt{R^2-\frac{b^2}{4}}
        \end{pmatrix}, \quad x^4=\begin{pmatrix}
            -\frac{b}{2} \\ 0 \\ -\sqrt{R^2-\frac{b^2}{4}}
        \end{pmatrix}.
    \end{equation*}
    Since $\norm{x^1-x^3}=D$, we obtain
    \begin{equation*}
    \begin{split}
         D^2&= \frac{a^2}{4}+\frac{b^2}{4}+R^2-\frac{a^2}{4}+ R^2-\frac{b^2}{4} + 2\sqrt{R^2-\frac{a^2}{4}}\sqrt{R^2-\frac{b^2}{4}}\\
         &=2R^2+ 2\sqrt{R^2-\frac{a^2}{4}}\sqrt{R^2-\frac{b^2}{4}},
    \end{split}
        \end{equation*}
        and therefore
         \begin{equation}
         \label{eq:ab}
          \sqrt{R^2-\frac{a^2}{4}}\sqrt{R^2-\frac{b^2}{4}}=\frac{D^2-2R^2}{2}.
    \end{equation}
    
    We know that the inball touches all facets of a simplex and that in our situation the incenter $c$ is situated on the $x_3$-axis due to symmetry reasons. Now, if we project $S$ onto the $x_1,x_3$--plane and the $x_2,x_3$--plane, the projections of the inball are circles with radius $r:=r(S)$. When these projections are overlaid, the two projected circles coincide  (\cf~\Cref{fig:specialsimplices}). Furthermore, since each projection direction is parallel to one of the two shorter edges and parallel to two different pairs of facets, after overlaying, 
    the circles touch the projections of all four facets.    
    \begin{figure}
    \centering
    \begin{tikzpicture}
        \tkzDefPoint(0,0){z}
          \tkzDefPoint(0,0.7){h1}
            \tkzDefPoint(0,-0.7){h2}
            \tkzDefPoint(0.85,0){h3}
        \tkzDefPoint(1,0){z1}
        \tkzDefPoint(0.8,-0.6){t3}
        \tkzDefPoint(-0.8,-0.6){t4}
        \tkzDefPoint(0.7,0.714){t1}
        \tkzDefPoint(-0.7,0.714){t2}
        \tkzDrawCircle[thick](z,t1)
         \tkzDefLine[orthogonal=through t1](z,t1) \tkzGetPoint{s1}
        \tkzDefLine[orthogonal=through t2](z,t2) \tkzGetPoint{s2}
         \tkzDefLine[orthogonal=through t3](z,t3) \tkzGetPoint{s3}
         \tkzDefLine[orthogonal=through t4](z,t4) \tkzGetPoint{s4}
          \tkzInterLL(t1,s1)(t2,s2)  \tkzGetPoint{m1}
          \tkzInterLL(t3,s3)(t4,s4)  \tkzGetPoint{m2}

        \tkzDefLine[orthogonal=through m1](z,m1) \tkzGetPoint{s5}
          \tkzDefLine[orthogonal=through m2](z,m2) \tkzGetPoint{s6}
        \tkzInterLL(m1,s5)(t3,m2)  \tkzGetPoint{x1}
          \tkzInterLL(m1,s5)(t4,m2)  \tkzGetPoint{x2}
          \tkzInterLL(m2,s6)(t1,m1)  \tkzGetPoint{x3}
          \tkzInterLL(m2,s6)(t2,m1)  \tkzGetPoint{x4}
          \tkzDrawPolygon(x1,x2,m2)
          \tkzDrawPolygon(x3,x4,m1)
          \tkzDrawSegment[thick, dblue](z,t1)
          \tkzLabelSegment[dblue, pos=0.6](z,t1){$r$}
          \tkzDrawSegment[dgreen](z,m1)
          \tkzDrawSegment[orange](z,m2)
          \tkzLabelSegment[dgreen, pos=0.2, left](z,m1){$h_1$}
           \tkzLabelSegment[orange, pos=0.8, left](m2,z){$h_2$}
           \tkzLabelAngle[pos=0.6](m2,x4,m1){$\alpha$}
 \tkzMarkAngle[size=1](m2,x4,m1)
    \tkzLabelAngle[pos=0.7](m1,x1,m2){$\beta$}
 \tkzMarkAngle[size=1](m1,x1,m2)
 \tkzLabelPoint[right](x1){$\bar{x}^1$}
 \tkzLabelPoint[left](x2){$\bar{x}^2$}
 \tkzLabelPoint[right](x3){$\bar{x}^3$}
 \tkzLabelPoint[left](x4){$\bar{x}^4$}
 \tkzLabelPoint[above](m1){$m^1$}
 \tkzLabelPoint[below](m2){$m^2$}
 \tkzLabelPoint[below right](z){$\bar{c}$}
        \tkzDrawPoints(z,t1,t2,t3,t4,m1,m2,x1,x2,x3,x4)
 \tkzDefPointBy[translation=from z to h1](x1) \tkzGetPoint{y1}
  \tkzDefPointBy[translation=from z to h1](x2) \tkzGetPoint{y2}
   \tkzDefPointBy[translation=from z to h2](x3) \tkzGetPoint{y3}
    \tkzDefPointBy[translation=from z to h2](x4) \tkzGetPoint{y4}
        \tkzDrawSegment[dgreen](y2,y1)
          \tkzDrawSegment[orange](y3,y4)
           \tkzLabelSegment[dgreen,above](y2,y1){$a$}
            \tkzLabelSegment[orange](y3,y4){$b$}
        \tkzDrawPoints[dgreen](y1,y2)
          \tkzDrawPoints[orange](y3,y4)
 \tkzDefPointBy[translation=from z to h3](x3) \tkzGetPoint{y5}
 \tkzDefPointBy[translation=from m2 to m1](y5) \tkzGetPoint{y6}
 \tkzDrawSegment[dblue](y5,y6)
           \tkzLabelSegment[dblue,right](y6,y5){$h$}
            \tkzDrawPoints[dblue](y5,y6)
          
    \end{tikzpicture}
    \caption{Proof of \Cref{lem:specialsimplices}: overlay of the projections of $c+r \B \optc S$ onto the $x_1,x_3$-plane and the $x_2,x_3$-plane.}
    \label{fig:specialsimplices}
\end{figure}
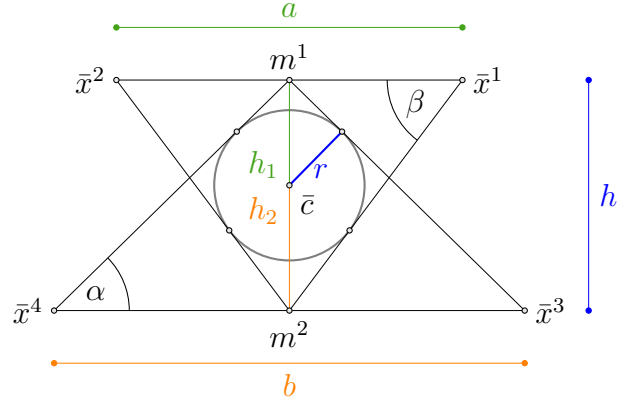

We define 
  \begin{equation*}
        \bar{x}^1:=\begin{pmatrix}
          \frac{a}{2} \\ \sqrt{R^2-\frac{a^2}{4}}
        \end{pmatrix},\quad  \bar{x}^2:=\begin{pmatrix}
              -\frac{a}{2} \\ \sqrt{R^2-\frac{a^2}{4}}
        \end{pmatrix}, \quad  \bar{x}^3:=\begin{pmatrix}
            \frac{b}{2} \\ -\sqrt{R^2-\frac{b^2}{4}}
        \end{pmatrix}, \quad  \bar{x}^4:=\begin{pmatrix}
            -\frac{b}{2} \\  -\sqrt{R^2-\frac{b^2}{4}}
        \end{pmatrix}, 
    \end{equation*}
   and 
    \begin{equation*}
       m^1:=\begin{pmatrix}
             0 \\ \sqrt{R^2-\frac{a^2}{4}}
        \end{pmatrix},\quad  m^2:=\begin{pmatrix}
          0 \\ -\sqrt{R^2-\frac{b^2}{4}}
        \end{pmatrix},
        \quad  \bar{c}:=\begin{pmatrix}
             0 \\  \bar{c}_3
        \end{pmatrix}. 
    \end{equation*}
   Then, $h:= \sqrt{R^2-\frac{a^2}{4}}+\sqrt{R^2-\frac{b^2}{4}}$ denotes the distance between 
   $m^1$ and $m^2$ and $h=h_1+h_2$, where $h_i$ denotes the distance from $m_i$ to $\bar{c}$, $i=1,2$. 
  Furthermore, we define $\alpha:=\angle m^2\bar{x}^4m^1$ and $\beta:=\angle m^1 \bar{x}^1 m^2$. 
    Then,  
    \begin{align}
       \label{eq:Dsimplex1} \frac{r}{h_1}&=\sin(\frac{\pi}{2}-\alpha)=\cos(\alpha),\\
        \label{eq:Dsimplex2}\frac{r}{h_2}&=\sin(\frac{\pi}{2}-\beta)=\cos(\beta),\\
        \label{eq:Dsimplex3} \cos(\beta)&=\frac{\frac{a}{2}}{\sqrt{h^2+\frac{a^2}{4}}},\\
         \label{eq:Dsimplex4}\cos(\alpha)&=\frac{\frac{b}{2}}{\sqrt{h^2+\frac{b^2}{4}}}, \text{ and }\\
         \label{eq:Dsimplex5}h^2&=D^2-\frac{a^2}{4}-\frac{b^2}{4}. 
    \end{align}
    From \eqref{eq:Dsimplex1} and \eqref{eq:Dsimplex2}, it follows $h=h_1+h_2=r\left(\frac{1}{\cos(\alpha)}+\frac{1}{\cos(\beta)}\right)$ and therefore using \eqref{eq:Dsimplex3}, \eqref{eq:Dsimplex4}, and \eqref{eq:Dsimplex5}
    \begin{align*}
        r&=\frac{h}{\frac{\sqrt{h^2+\frac{a^2}{4}}}{\frac{a}{2}}+\frac{\sqrt{h^2+\frac{b^2}{4}}}{\frac{b}{2}}}\\
        &=\frac{hab}{2a\sqrt{h^2+\frac{b^2}{4}}+2b\sqrt{h^2+\frac{a^2}{4}}}\\
        &=\frac{\left(\sqrt{R^2-\frac{a^2}{4}}+\sqrt{R^2-\frac{b^2}{4}}\right)ab}{2a\sqrt{D^2-\frac{a^2}{4}}+2b\sqrt{D^2-\frac{b^2}{4}}}.
    \end{align*}
\end{proof}

The following lemma further specifies the simplices that come into question for the minimal inradius.
Its rather technical proof can be found in the Appendix.
\begin{lemma}
\label{lem:decreasingfct}
Out of the simplices optimally contained in $\B$ with four edges of length $D\in\left[\sqrt{\frac{8}{3}},\sqrt{3}\right)$ and two opposing edges of lengths $a\leq b\leq D$ those with five edges of length $D$ (\ie~$b=D$) uniquely minimize the inradius. For such a simplex $S$, the inradius is given as
\begin{equation*}
    r(S)=\frac{D^2\sqrt{3-D^2}}{4\sqrt{3-D^2}-\sqrt{3}(D^2-4)}.
\end{equation*}
\end{lemma}

In the next lemma, we prove inequality \eqref{eq:NewInequality} for simplices. 
\begin{lemma}
    \label{lem:optinradius}
    For every three-dimensional simplex $S$ optimally contained in $R(S)\B$ with $D(S)\in\left[\sqrt{\frac{8}{3}}R(S), \sqrt{3}R(S)\right)$ we have
    \begin{equation*}
        r(S)\geq \frac{D(S)^2\sqrt{3R(S)^2-D(S)^2}}{4R(S)\sqrt{3R(S)^2-D(S)^2}+\sqrt{3}(4R(S)^2-D(S)^2)}.
    \end{equation*}
    Equality is attained if and only if $S$ has five edges of length $D(S)$.
\end{lemma}
\begin{proof}
By \Cref{lem:optinradius1}, the inradius of $S$ is at 
least the one of a simplex $S'$ with the property that each vertex is adjacent to at least two edges of length $D$ and $R(S')=R(S)$ and $D(S')=D(S)$. 

\Cref{lem:decreasingfct} shows that the expression in \Cref{lem:specialsimplices} is uniquely smallest if five edges have length $D$ and that in this case it equals the right-hand side of our claim. 

\end{proof}
    
In the planar Euclidean case, the corresponding left boundary was filled by isosceles triangles. Thus, we call the three-dimensional simplices attaining equality here, with only one shorter edge, \textit{isosceles} as well.

We are now ready to prove the main theorem.

\begin{proof}[Proof of \Cref{thm:fulldiagram}]
As already mentioned, 
it suffices to prove \eqref{eq:NewInequality} as the four equalities in \eqref{eq:4eineqsKnown} are already well-known. The first and fourth follow from the definitions, the second was shown by Jung \cite{Jung1901} (\cf~\eqref{eq:jungeucl}), and the third is shown in \cite{vrecica1981note} (it is also called the concentricity inequality, \cf~\cite{CompleteSimpl}). 
It also follows from the definition that $r(K)\geq 0$. 
To show \eqref{eq:NewInequality}, let $K\in\CK^3$ with $K\optc R(K)\B$ and $D(K)<\sqrt{3}R(K)$. If two or three touching points are enough to characterize the optimal containment via \Cref{prop:opt}, then there exists a subdimensional subset $K'$ of $K$ with $R(K')=R(K)$, and it follows from Jung's inequality that $D(K)\geq D(K')\geq \sqrt{3}R(K)$ (\cf~\Cref{rem:sqrt3}).

To achieve $D(K)<\sqrt{3}R(K)$ one needs four affinely independent touching points $p^1,p^2,p^3,p^4 \in \bd(K)\cap R(K)\S$. Defining $S:=\conv(\set{p^1,p^2,p^3,p^4})$, we obtain $R(S)=R(K)$, $r(S)\leq r(K)$, and $D(S)\leq D(K)<\sqrt{3}$. Now, we apply \Cref{lem:optinradius} and obtain
    \begin{align*}
        r(K)&\geq r(S)\geq \frac{D(S)^2\sqrt{3R(S)^2-D(S)^2}}{4R(S)\sqrt{3R(S)^2-D(S)^2}-\sqrt{3}(D(S)^2-4R(S)^2)}\\
        &\geq \frac{D(K)^2\sqrt{3R(K)^2-D(K)^2}}{4R(K)\sqrt{3R(K)^2-D(K)^2}-\sqrt{3}(D(K)^2-4R(K)^2)}
    \end{align*}
by using the fact that the right side of the inequality in \Cref{lem:optinradius} is decreasing in the diameter. 
If $K$ is not a simplex $r(K)>r(S)$ by \Cref{lem:inballsimplex} and for simplices we know from \Cref{lem:optinradius} that equality is attained exactly for our isosceles simplices. 

Since \eqref{eq:NewInequality} is continuous, it completely describes the left side of the diagram.

To show that \eqref{eq:4eineqsKnown} and \eqref{eq:NewInequality} together form a complete system of inequalities for the $(r,D,R)$-diagram, we need to describe the bodies that fill the induced boundaries of $f(\CK^3)$. Once this is settled, Proposition \ref{prop:starshape} would tell us that the bounded space contained between the provided boundaries is completely filled by images $f(K)$ of convex bodies $K$. 

First, the boundary $R(K)\leq \sqrt{\frac{3}{8}}D(K)$ can be filled by bodies between the regular 3-simplex $T$ and any (Scott-)completion of $T$ (i.e.  a complete set containing $T$ with the same diameter as $T$, \eg~the Meissner bodies \cite{meissner1912}). 
\Cref{prop:starshape} is sufficient to fill the boundaries induced by $D(K)\leq 2R(K)$ and $r(K)+R(K)\leq D(K)$ (by simply considering the convex combinations of a line segment and a Meissner body with the Euclidean ball, respectively). 
Isosceles triangles fill the boundary induced by $0\leq r(K)$ \cite{santalo}.
Finally, isosceles simplices attain equality for the new inequality. 

\end{proof}


\newpage

\appendix 

\section{} \label{sec:appendix}

\begin{proof}[Proof of \Cref{lem:decreasingfct}]
We calculated the inradius of simplices with only one pair of opposing edges being shorter than the diameter in \Cref{lem:specialsimplices}:
 \begin{equation}
 \label{eq:inradiussimplex1}
       \frac{\left(\sqrt{1-\frac{a^2}{4}}+\sqrt{1-\frac{b^2}{4}}\right)ab}{2a\sqrt{D^2-\frac{a^2}{4}}+2b\sqrt{D^2-\frac{b^2}{4}}}=  \frac{\sqrt{1-\frac{a^2}{4}}+\sqrt{1-\frac{b^2}{4}}}{\frac{2}{b}\sqrt{D^2-\frac{a^2}{4}}+\frac{2}{a}\sqrt{D^2-\frac{b^2}{4}}}.
    \end{equation}
    From 
\begin{equation*}
          \sqrt{1-\frac{a^2}{4}}\sqrt{1-\frac{b^2}{4}}=\frac{D^2-2}{2}
    \end{equation*}
    we obtain 
    \begin{equation}
    \label{eq:abyb}
        \frac{a^2}{4}=1-\frac{(D^2-2)^2}{4}\cdot \frac{1}{1-\frac{b^2}{4}}.
    \end{equation}
   Now, $a=b$ if and only if $\frac{b^2}{4}=\frac{4-D^2}{2}$. Thus, $a\leq b \le D$, implies $\frac{b^2}{4}\in \left[\frac{4-D^2}{2}, \frac{D^2}{4}\right]$. 
  Inserting \eqref{eq:abyb} into \eqref{eq:inradiussimplex1},
replacing $\frac{b^2}{4}$ by $x$, and simplifying yields 
\begin{equation}
\label{eq:appinradius}
\begin{split}
     &\frac{\sqrt{\frac{(D^2-2)^2}{4}\cdot\frac{1}{1-x}}+\sqrt{1-x}}{\sqrt{\frac{1}{x}}\sqrt{D^2-1+\frac{(D^2-2)^2}{4}\cdot\frac{1}{1-x}}+\sqrt{\frac{4(1-x)}{4(1-x)-(D^2-2)^2}}\sqrt{D^2-x}}\\
     &=\frac{\frac{(D^2-2)}{2}+1-x}{\sqrt{\frac{1}{x}}\sqrt{(D^2-1)(1-x)+\frac{(D^2-2)^2}{4}}+(1-x)\sqrt{\frac{4(D^2-4)}{4(1-x)-(D^2-2)^2}}}\\
     &=\frac{\frac{D^2}{2}-x}{
        \sqrt{\frac{1}{x}} \sqrt{x(1-D^2)+\frac{D^4}{4}} +(1-x)\sqrt{\frac{4(D^2-x)}{4(D^2-x)-D^4}} }.
\end{split}
\end{equation}
  Next, we show that \eqref{eq:appinradius} is strictly decreasing for $x\in\left[\frac{4-D^2}{2},\frac{D^2}{4}\right]$, and therefore, the smallest inradius is attained if five edges have length $D$. 
    For better readability, we replace $D^2$ by $d$ and consider the function  \begin{equation*}
    \begin{split}
        &f:\left[\frac{4-d}{2}, \frac{d}{4}\right]\to\R,\\
       & f(x)=\frac{\frac{d}{2}-x}{
        \sqrt{\frac{1}{x}} \sqrt{x(1-d)+\frac{d^2}{4}} +(1-x)\sqrt{\frac{4(d-x)}{4(d-x)-d^2}} }
    \end{split}
         \end{equation*}
         for $d\in\left[\frac{8}{3},3\right)$. Let us first compute the derivative:

\begin{align*}
    f'(x)&= \bigg[-\bigg(
        \sqrt{\frac{1}{x}} \sqrt{x(1-d)+\frac{d^2}{4}} +(1-x)\sqrt{\frac{4(d-x)}{4(d-x)-d^2}} \bigg)\\
        & -(\frac{d}{2}-x)\bigg(\frac{-d^2}{4x\sqrt{x}\sqrt{4x+d^2-4dx}}-\sqrt{\frac{4(d-x)}{4(d-x)-d^2}}\\
        &+ \frac{(1-x)d^2}{\sqrt{d-x}\sqrt{4(d-x)-d^2}(4(d-x)-d^2)}\bigg)
        \bigg] \\
    &\bigg/ \left(
        \sqrt{\frac{1}{x}} \sqrt{x(1-d)+\frac{d^2}{4}} +(1-x)\sqrt{\frac{4(d-x)}{4(d-x)-d^2}} \right)^2.
\end{align*}

The goal is to show that $f'(x)< 0, x\in\left(\frac{4-D^2}{2},\frac{D^2}{4}\right] $. The denominator is obviously non-negative. So, we only need to consider the numerator, which we split into two parts. With $g$ we denote the first and third summand and the remaining parts are called $h$: 
\begin{align*}
\label{eq:ap1}
    g(x)&=-\sqrt{\frac{1}{x}} \sqrt{x(1-d)+\frac{d^2}{4}}-(\frac{d}{2}-x)\frac{-d^2}{4x\sqrt{x}\sqrt{4x+d^2-4dx}},\\
    h(x)&=-(1-x)\sqrt{\frac{4(d-x)}{4(d-x)-d^2}}-(\frac{d}{2}-x)\bigg(-\sqrt{\frac{4(d-x)}{4(d-x)-d^2}}\\ &+ \frac{(1-x)d^2}{\sqrt{d-x}\sqrt{4(d-x)-d^2}(4(d-x)-d^2)}\bigg).
\end{align*}
Now, we show that $g(x)$ and $h(x)$ are strictly decreasing in the intervall $\left(\frac{4-d}{2},\frac{d}{4}\right]$. 
Simplyfing yields: 
\begin{equation*}
\label{eq:ap3}
\begin{split}
      g(x)&=\frac{-\sqrt{4x+d^2-4dx}}{2\sqrt{x}}+(\frac{d}{2}-x)\frac{d^2}{4x\sqrt{x}\sqrt{4x+d^2-4dx}}\\
      &=\frac{2x(4x+d^2-4dx)+d^2(\frac{d}{2}-x)}{4x\sqrt{x}\sqrt{4x+d^2-4dx}}\\
      &=\frac{(16d-16)x^2-6d^2x+d^3}{8x\sqrt{x}\sqrt{4x+d^2-4dx}}.
\end{split}
\end{equation*}
The derivative is 
\begin{equation*}
\label{eq:ap4}
\begin{split}
      g'(x)&=\frac{-d^2(3d^3-22d^2x-32x^2+16dx(1+2x))}{16x^2\sqrt{x}\sqrt{4x+d^2-4dx}(4x+d^2-4dx)}\\
      &=\frac{-d^2(d-2x)(3d^2-16(d-1)x)}{16x^2\sqrt{x}\sqrt{4x+d^2-4dx}(4x+d^2-4dx)}.
      \end{split}
\end{equation*}
Since $d-2x> 0$ and $3d^2-16(d-1)x> 0$ for $x\in \left(\frac{4-d}{2},\frac{d}{4}\right]$ and $d\in\left[\frac{8}{3},3\right)$, we obtain $g'(x)\leq 0$. 

We simplify $h$ as well:

\begin{equation*}
\label{eq:ap5}
\begin{split}
      h(x)&=\frac{(d-2)\sqrt{d-x}}{\sqrt{4(d-x)-d^2}}- \frac{(\frac{d}{2}-x)(1-x)d^2}{\sqrt{d-x}\sqrt{4(d-x)-d^2}(4(d-x)-d^2)}\\
      &=\frac{(d-2)(d-x)(4(d-x)-d^2)-(\frac{d}{2}-x)(1-x)d^2}{\sqrt{d-x}\sqrt{4(d-x)-d^2}(4(d-x)-d^2)}.
\end{split}
\end{equation*}
The derivative is
\begin{equation*}
\label{eq:ap6}
\begin{split}
      h'(x)=\frac{d^2((-6d^2+16d-32)x^2+(11d^3-50d^2+48d)x-4d^4+17d^3-16d^2)}{4\sqrt{d-x}(d-x)\sqrt{4(d-x)-d^2}(4(d-x)-d^2)^2}.
      \end{split}
\end{equation*}
The denominator and $d^2$ are non-negative. Thus, let us consider the quadratic function 
\begin{equation*}
    q(x)=(-6d^2+16d-32)x^2+(11d^3-50d^2+48d)x-4d^4+17d^3-16d^2.
\end{equation*}
Since $d\geq \frac{8}{3}$, we have
\begin{align*}
    -6d^2+16d-32=-6\left(d-\frac{4}{3}\right)^2-\frac{64}{3}<0. 
\end{align*}
Next, we show that $q(\frac{1}{2})< 0$ and $q'(\frac{1}{2})<0$ for every choice of $d\in\left[\frac{8}{3},3\right)$. Since $x\geq \frac{4-d}{2}\geq \frac{1}{2}$, this implies $q(x)< 0$ for all $x\in\left(\frac{4-d}{2},\frac{d}{4}\right]$.

For $d\in \left[\frac{8}{3},3\right]$, we have
\begin{align*}
    q(\tfrac{1}{2})&=(-6d^2+16d-32)\frac{1}{4}+(11d^3-50d^2+48d)\frac{1}{2}-4d^4+17d^3-16d^2\\
    &=-4d^4+\frac{45}{2}d^3-\frac{89}{2}d^2+28d-8\\
    &=\frac{1}{2}d^2(-8d^2+45d-68)+d\left(28-\frac{21}{2}d\right)-8\\
    &=\frac{1}{2}d^2\left(-8\left(d-\frac{45}{2}\right)^2-\frac{151}{32}\right)+d\left(28-\frac{21}{2}d\right)-8<0.
\end{align*}

Furthermore,
\begin{equation*}
    \begin{split}
        q'(\tfrac{1}{2})&= -6d^2+16d-32 +11d^3-50d^2+48d\\
        &=11d^3-56d^2+48d-32\\
        &=d(d-4)(11d-12)-32<0.
    \end{split}
\end{equation*}

Thus, we have shown that $h(x)+g(x)< g(\frac{4-d}{2})+h(\frac{4-d}{2})$.
We compute $g(\frac{4-d}{2})$ and $h(\frac{4-d}{2})$.
\begin{align*}
    g\left(\frac{4-d}{2}\right)&=\frac{4(d-1)(4-d)^2-3d^2(4-d)+d^3}{2\sqrt{2}(4-d)\sqrt{4-d}\sqrt{2(4-d)(1-d)+d^2}}\\
    &=\frac{8(d-2)^3}{2\sqrt{2}(4-d)\sqrt{(4-d)(3d-4)(d-2)}}\\
    &=\frac{2\sqrt{2}(d-2)^3}{(4-d)\sqrt{(4-d)(3d-4)(d-2)}}.
\end{align*}
For $h$, using $d-x=\frac{3d-4}{2}$, $1-x=\frac{d-2}{2}$, $\frac{d}{2}-x=d-2$, and $4(d-x)-d^2=(d-2)(4-d)$, we obtain 
\begin{align*}
     h\left(\frac{4-d}{2}\right)&=\frac{(d-2)^2(3d-4)(4-d)-d^2(d-2)^2}{2\sqrt{\frac{3d-4}{2}}\sqrt{(d-2)(4-d)}(d-2)(4-d)}\\
     &=\frac{-4(d-2)^4}{\sqrt{2}(d-2)(4-d)\sqrt{(4-d)(3d-4)(d-2)}}\\
    &=\frac{-2\sqrt{2}(d-2)^3}{(4-d)\sqrt{(4-d)(3d-4)(d-2)}}
\end{align*}

Together, we obtain $h(x)+g(x)< g(\frac{4-d}{2})+h(\frac{4-d}{2})=0$. Therefore, $f'(x)< 0$ for $x\in\left(\frac{4-d}{2},\frac{d}{4}\right]$ and $f$ is strictly decreasing. 

The inradius of a simplex $S$ optimally contained in $\B$ with at least five edges of length $D$ 
\begin{align*}
    r(S)&=f\left(\frac{D^2}{4}\right)=\frac{\frac{D^2}{2}-\frac{D^2}{4}}{\frac{2}{D}\sqrt{\frac{D^2}{4}(1-D^2)+\frac{D^4}{4}}+\left(1-\frac{D^2}{4}\right)\sqrt{\frac{3D^2}{3D^2-D^4}}}\\
    &=\frac{D^2}{4+(4-D^2)\sqrt{\frac{3}{3-D^2}}}=\frac{D^2\sqrt{3-D^2}}{4\sqrt{3-D^2}-\sqrt{3}(D^2-4)}
\end{align*}
and they are the unique minimizers.
\end{proof}

\newpage

René Brandenberg -- 
Technical University of Munich, School of Computation, Information and Technology, Department of Mathematics, Germany. \\
\textbf{rene.brandenberg@tum.de}

Bernardo González Merino -- 
University of Murcia, Faculty of Computer Science, Department of Engineering and Tecnology of Computers, Area of Applied Mathematics, Spain.\\
\textbf{bgmerino@um.es}

Mia Runge -- 
Technical University of Munich, School of Computation, Information and Technology, Department of Mathematics, Germany. \\
\textbf{mia.runge@tum.de}

\end{document}